\documentclass[reqno]{amsart}
\usepackage{amsmath,amssymb,graphics,epsfig,color,cite,hyperref,listing}
\usepackage[norelsize,ruled,vlined]{algorithm2e}

\newtheorem{theorem}{Theorem}[section]
\newtheorem{proposition}[theorem]{Proposition}
\newtheorem{lemma}[theorem]{Lemma}
\newtheorem{corollary}[theorem]{Corollary}
\theoremstyle{remark}
\newtheorem{remark}[theorem]{Remark}
\theoremstyle{definition}
\newtheorem{definition}[theorem]{Definition}

\numberwithin{equation}{section}
\numberwithin{theorem}{section}

\def\be{\begin{equation}}
\def\ee{\end{equation}}
\def\bp{\begin{pmatrix}}
\def\ep{\end{pmatrix}}
\def\bea{\begin{eqnarray}}
\def\eea{\end{eqnarray}}

\newcommand{\mc}[1]{{\mathcal #1}}
\newcommand{\bb}[1]{{\mathbb #1}}

\newcommand{\Sp}{{\mathcal P}}
\newcommand{\C}{\mathbb C}

\newcommand{\M}{{\mathcal M}}

\newcommand{\K}{\mathbb K}
\newcommand{\R}{\mathbb R}

\newcommand{\bigO}{\mathcal{O}}
\renewcommand{\S}{N}
\def\tp{{\rm T}}
\def\iu{\imagunit}
\def\e{{\rm e}}
\def\d{{\rm d}}
\def\i{{\rm i}}

\def\eps{\varepsilon}
\def\bigo{{\mathcal O}}
\def\l{\lambda}
\def\tp{{\rm T}}

\def\pin{^\dagger}

\newcommand{\etalchar}[1]{$^{#1}$}
\newcommand{\Lameps}{\Lambda_\varepsilon}

\renewcommand{\Re}{\mathrm{Re}}
\renewcommand{\Im}{\mathrm{Im}}

\newcommand{\imagunit}{{\bf i}}

\newcommand{\conj}[1]{\overline{#1}}

\newcommand{\lan}{\langle}
\newcommand{\ran}{\rangle}
\newcommand{\cR}{\mathcal{R}}
\newcommand{\Stru}{{\rm Str}}
\begin{document}

\title[Differential equations for defectivity measures]{Differential equations for real-structured  (and unstructured) defectivity measures}

\author[P.\ Butt\`a]{P.\ Butt\`a}\address{Dipartimento di Matematica, SAPIENZA Universit\`a di Roma, P.le Aldo Moro 5, 00185 Roma, Italy}
 \email{butta@mat.uniroma1.it}
\author[N.\ Guglielmi] {N.\ Guglielmi}\address{Dipartimento di Ingegneria Scienze Informatiche e Matematica, Universit\`{a} degli studi di L'Aquila, Via Vetoio - Loc.~Coppito, I-67100 L'Aquila, Italy} 
\email{guglielm@univaq.it}
\author[M.\ Manetta]{M.\ Manetta}\address{Dipartimento di Ingegneria Scienze Informatiche e Matematica, Universit\`{a} degli studi di L'Aquila, Via Vetoio - Loc.~Coppito, I-67100 L'Aquila, Italy} 
\email{manetta@univaq.it}
\author[S.\ Noschese]{S.\ Noschese}\address{Dipartimento di Matematica, SAPIENZA Universit\`a di Roma, P.le Aldo Moro 5, 00185 Roma, Italy}
\email{noschese@mat.uniroma1.it}

\subjclass[2010]{15A18, 65K05}

\keywords{Pseudospectrum, structured pseudospectrum, low rank differential equations, defective eigenvalue, distance to defectivity, Wilkinson problem.}

\begin{abstract}
Let $A$ be either a complex or real matrix with all distinct eigenvalues. We propose a new method for the computation of both the unstructured and the real-structured (if the matrix is real) distance $w_{\mathbb K}(A)$ (where ${\mathbb K}=\mathbb C$ if general complex matrices are considered and ${\mathbb K} ={\mathbb R}$ if only real matrices are allowed) of the matrix $A$ from the set of defective matrices, that is  the set of those matrices with at least a multiple eigenvalue with algebraic multiplicity larger than its geometric multiplicity. For $0 < \varepsilon \le w_{\mathbb K}(A)$, this problem is closely related to the computation of the most ill-conditioned $\varepsilon$-pseudoeigenvalues of $A$, that is points in the $\varepsilon$-pseudospectrum of $A$ characterized by the highest condition number. The method we propose couples a system of differential equations on a low rank (possibly structured) manifold which computes the $\varepsilon$-pseudoeigenvalue of $A$ which is closest to coalesce, with a fast Newton-like iteration aiming to determine the minimal value $\varepsilon$ such that such an $\varepsilon$-pseudoeigenvalue becomes defective. The method has a local behaviour; this means that in general we find upper bounds for $w_{\mathbb K}(A)$. However, they usually provide good approximations, in those (simple) cases where we can check this. The methodology can be extended to a structured matrix where it is required that the distance is computed within some manifold defining the structure of the matrix. In this paper we extensively examine the case of real matrices but we also consider pattern structures. As far as we know there do not exist methods in the literature able to compute such distance. 
\end{abstract}

\maketitle
\thispagestyle{empty}

\section{Introduction}
\label{sec:1}

Let $A \in \K^{n,n}$ be a complex ($\K=\C$) or real ($\K = \R$) matrix with all distinct eigenvalues. 
We are interested to compute the following distance,
\begin{eqnarray}
w_\K(A) & = & \inf \{ \| A - B \|_F \colon B \in \K^{n,n} \ \mbox{is defective} \}\;,
\label{eq:dist}
\end{eqnarray}
where $\| \cdot \|_F$ is the Frobenius norm (when $\K = \C$ this turns out to be
equivalent to consider the $2$-norm).
We recall that a matrix is defective if its Jordan canonical form has at least a non diagonal block associated to an eigenvalue $\lambda$.
Let us introduce the $\eps$-pseudospectrum of $A$,
\begin{eqnarray}
    \Lameps^{\K}(A) & = & 
		\{ \lambda \in \C \colon \lambda \in \Lambda\left(A+E\right) \ \mbox{for some} \
     E \in \K^{n\times n} \ \mbox{with} \ \|E\|_F \le \eps\}\;,
\label{eq:Lameps}
\end{eqnarray}
where we refer to the classical monograph by Trefethen and Embree \cite{TrEm05} for an extensive treatise.
In his seminal paper \cite{Wil65}, Wilkinson defined the condition number of a simple eigenvalue as
\[
k(\lambda) = \frac{1}{|y^H x|}\;, \qquad \mbox{$y$ and $x$ left and right eigenvectors with $\| x \|_2 = \| y \|_2 = 1$}\;.
\]
Observing that $k(\lambda) = +\infty$ for a defective
eigenvalue since $y^H x = 0$, the search of the closest defective matrix can be pursued by 
looking for the minimal value $\eps$ such that there exists $z \in \Lameps^{\K}(A)$ with $y^Hx=0$,
where $z \in \Lambda(A + E)$ for some $E$ of norm $\eps$, being $y$ and $x$ the normalized 
left and right eigenvectors associated to $z$.

The distance $w_\C(A)$ was introduced by Demmel \cite{Dem83} in his well-known PhD thesis and has been studied by several authors, not only for its theoretical interest but also for its practical one (see, e.g., \cite{Ala06} for the bounds on $w_\K(A)$ presented in the literature, \cite{AFS13} and references therein for physical applications). An interesting formula for computing this distance in the $2$-norm has been given by Malyshev \cite{M99}.

In the recent and very interesting paper by Alam, Bora, Byers and Overton \cite{ABBO11}, which provides an extensive historical analysis of the problem, the authors have shown that when $\K = \C$ the infimum in (\ref{eq:dist}) is actually a minimum. 
Furthermore, in the same paper, the authors  have proposed a computational approach to approximate the nearest defective matrix by a variant of Newton's method. Such method is well suited to dense problems of moderate size $n$ and - even for a real matrix - computes a nearby defective complex matrix, that is a minimizer for (\ref{eq:dist}) in $\C^{n\times n}$.
A recent fast algorithm has been given in \cite{AFS13}, which is based on an extension
of the implicit determinant method proposed in \cite{PS05}. This method also deals with the unstructured and 
provides a nearby defective complex matrix.

The aim of this paper is that of providing a different approach to the approximation of $w_\K(A)$ for both $\K=\R$ and $\K=\C$, which
may be extended to the approximation of more general structured distances, that is, for example, when restricting the admissible perturbations of $A$ to the set of matrices with a prescribed nonzero pattern.
However, a rigorous analysis of general structures is beyond the scope of this paper and we limit the discussion to complex and
real perturbations.

The methodology we propose is splitted in two parts. First, for a given $\eps < w_\K(A)$ we are
interested to compute the following quantity,
\begin{eqnarray}
\hskip -3mm
r(\eps) = \min\{ |y^Hx| \colon \mbox{$y$ and $x$ left/right eigenvectors associated to} \
z \in \Lameps^{\K}(A) \}\;.
\label{eq:condmax}
\end{eqnarray}
Secondly, we are interested to find the smallest solution to the equation $r(\eps) = 0$, or more in general, possibly introducing a small threshold $\delta \ge 0$, to the equation
\begin{eqnarray}
r(\eps)=\delta\;. 
\label{eq:r}
\end{eqnarray}
Note that - in this second case - we obtain anyway, as a byproduct of our method, an estimate for a solution of
$r(\eps) = 0$. 
We remark that any solution to $r(\eps) = 0$ gives in general an upper bound to $w_\K(A)$.

By the results of Alam and Bora \cite{AB05}, we deduce that the distance $w_\C(A)$ we
compute is the same one obtains by replacing the Frobenius norm by the $2$-norm.
Instead, for the real case, the distance we compute is in general larger than the
corresponding distance in the $2$-norm.

The paper is organized as follows. In Section \ref{sec:2} we analyze the case of general complex
perturbations and in Section \ref{sec:3} we derive a system of differential equations which form 
the basis of our computational framework. By a low rank property of the stationary points of the
system of ODEs, which identify the matrices which allow to compute approximations of (\ref{eq:dist}), 
we consider the projected system of ODEs on the corresponding low rank manifold in Section \ref{sec:4} 
and prove some peculiar results of the corresponding flow. In Section \ref{sec:5} we pass to consider 
the case of real matrices with real perturbations and obtain a new system of ODEs for the computation 
of (\ref{eq:condmax}), which are discussed in Section \ref{sec:6}. In Section \ref{sec:7} we present 
some theoretical results which allow us to obtain a fast method to solve (\ref{eq:r}) and compute 
approximations to $w_\K(A)$. In the same section we present the complete algorithm. Afterwards, in 
Section \ref{sec:8} we focus our attention on a few implementation issues and in Section \ref{sec:9}  
show some numerical examples. Finally in Section \ref{sec:10} we conclude the paper by providing an 
extension of the method to matrices with a prescribed sparsity pattern.   
  
\section{The complex case}
\label{sec:2}

We denote by $\|A\|_F = \sqrt{\lan A,A\ran}$ the Frobenius norm of the matrix 
$A\in  \bb C^{n\times n}$, where, for any given $A,B\in \bb C^{n\times n}$, 
$\lan A,B\ran = \mathrm{trace}(A ^HB)$.

We also need the following definition.

\begin{definition}
\label{def:groupinv}
Let $M$ be a singular matrix with a simple zero eigenvalue.
The \emph{group inverse} (reduced resolvent) of $M$, denoted $M^\#$, is the unique matrix $G$ satisfying $MG =GM$, $GMG=G$ and $MGM=M$.
\end{definition}

Given a matrix function $M(t) \in \bb C^{n\times n}$, smoothly depending on the real parameter $t$, we recall results concerning derivatives of right and left eigenvector $x(t)$ and $y(t)$, respectively, associated to a simple eigenvalue $\lambda(t)$ of $M(t)$. We denote by $G(t)$ the group-inverse of $M(t)-\lambda(t)I$ and assume $x(t),y(t)$ smoothly depending on $t$ and such that $\|x(t)\|_2=\|y(t)\|_2=1$. In the sequel, we shall often omit the explicit dependence on $t$ in the notation. The following expressions for the derivatives can be found in \cite[Theorem 2]{MS88},
\begin{equation}
\label{ms}
\dot x = x ^HG\dot M x x - G \dot M x\;, \qquad\quad  \dot y ^H = y ^H \dot M G y y ^H - y ^H \dot M G\;.
\end{equation}

Given $A\in\bb C^{n\times n}$ let $\lambda_0$ be a simple eigenvalue of $A$. We denote by $\mc S_1 = \{E\in\bb C^{n\times n}\colon \|E\|_F=1\}$ the unitary hyper-sphere in $\bb C^{n\times n}$. By continuity, there exists $\eps_0>0$ such that for any $\eps\in [0,\eps_0]$ and $E\in \mc S_1$ the matrix $A+\eps E$ has a simple eigenvalue $\lambda$ close to $\lambda_0$.

\begin{lemma}
\label{lem:2.1}
Given $\eps\in (0,\eps_0]$ let
\be
\label{S}
S = yy ^HG ^H+G ^Hxx ^H\;, 
\ee
being $G$ the group-inverse of $A + \eps E -\lambda I$, whose left and right null vectors are $y$ and $x$ (recall we assume $\|x\|_2=\|y\|_2=1$, and note that $y ^H x \ne 0$ as $\lambda$ is simple). Then, for any smooth path $E=E(t)\in \mc S_1$ we have,
\begin{equation}
\label{dy*x}
\frac{\d}{\d t} |y ^Hx| = \eps\,|y ^Hx| \, \Re\,\big\lan \dot E,S\big\ran\;.
\end{equation}
In particular, the steepest descent direction is given by
\be
\label{argmin}
\mathrm{arg} \min_{\substack{D\in\mc S_1 \\ \Re\lan D,E\ran=0}} \Re\big\lan D, S\big\ran = - \mu \big (S-\Re\lan E,S\ran E\big)\;,  
\ee
where $\mu>0$ is such that the right-hand side has unit Frobenius norm.
\end{lemma}

\begin{proof}
By \eqref{ms} and using that $Gx=0$, $y ^HG=0$, we get, 
\[
\begin{split}
\frac{\d}{\d t} |y ^Hx|^2 & = 2\,\Re\big\{\overline{y ^Hx}\, (\dot y ^H x + y ^H\dot x)\big\} = 2\eps\, \Re\big\{x ^Hy \, (y ^H\dot E G y y ^Hx + y ^Hx ^HG\dot E x x) \big\}  \\ & = 2\eps\,|y ^Hx|^2 \,\Re\big\{y ^H\dot E G y + x ^HG\dot E x\big\} \\ & = 2\eps\,|y ^Hx|^2 \, \Re \,\mathrm{trace}\big(\dot E ^H yy ^HG ^H+\dot E ^H G ^Hxx ^H\big) \;,
\end{split}
\]
that is, recalling the definition \eqref{S}, $\frac{\d}{\d t} |y ^Hx|^2 = 2\eps\,|y ^Hx|^2 \, \Re\,\big\lan \dot E,S\big\ran$, from which \eqref{dy*x} follows using that $\frac{\d}{\d t} |y ^Hx|  = \frac 12  |y ^Hx| ^{-1}  \frac{\d}{\d t} |y ^Hx|^2$. The steepest descent direction is then given by the solution of the variational problem \eqref{argmin}.
\end{proof}

\begin{remark}\rm
\label{rem:0}
A consequence of Lemma \ref{lem:2.1} is that the gradient $\nabla_E \;y^Hx$ of the function $E\mapsto y^Hx$ is proportional to $S$. 
\end{remark}

Let ${ \mc M}_2$ be the manifold of $n\times n$ matrices of rank-$2$. Any $E \in { \mc M}_2$ can be (non uniquely) represented in the form $E = U T V ^H$, where $U,V \in \C^{n\times 2}$ have orthonormal columns and $T \in \C^{2\times2}$ is nonsingular. We will use instead a unique decomposition in the tangent space: every tangent matrix $\delta E \in T_E{ \mc M}_2$ is of the form,
\be
\label{p:1}
\delta E = \delta U T V ^H + U \delta T V ^H + U T \delta V ^H\;,
\ee
where $\delta T\in \bb C^{2\times 2}$, and $\delta U, \delta V \in \bb C^{n\times 2}$ are such that
\be
\label{p:2}
U ^H\delta U = 0\;, \qquad V ^H\delta V = 0\;.
\ee
This representation is discussed in \cite{KL07} for the case of real matrices, but the extension to the case of complex matrices is straightforward. 

Under the assumptions above, the orthogonal projection of a matrix $Z\in \bb C^{n\times n}$ onto the tangent space $T_E{ \mc M}_2$ is given by
\be
\label{p:3}
P_E(Z) = Z - \left( I - U U ^H\right) Z \left( I - V V ^H \right)\;.
\ee

\section{System of ODEs}
\label{sec:3}

Let $A$, $\lambda_0$, and $\eps_0$ as before. The following theorem characterizes an evolution onto $\mc S_1$ governed by the steepest descent direction of $|y ^Hx|$, where $x$ and $y$ are unit-norm right and left eigenvectors, respectively, associated to the simple eigenvalue $\lambda$. 

\begin{theorem}
\label{thm:1}
Given $\eps\in (0,\eps_0]$, consider the differential system,
\be
\label{odefull}
\dot E = - S + \Re\lan E,S\ran E\;, \qquad E\in \mc S_1\;,
\ee
where $S$ is defined in \eqref{S}.
\begin{enumerate}
\item[1)] The right-hand side of \eqref{odefull} is antiparallel to the projection onto the tangent space $\mc T_E\mc S_1$ of the gradient of $|y ^Hx|$. More precisely,
\be
\label{decr}
\frac{\d}{\d t} |y ^Hx| = - \eps\, |y ^Hx| \, \|S-\Re\lan E,S\ran E\|_F^2\;.
\ee
In particular, $\frac{\d}{\d t} |y ^Hx| =0$ if and only if $\dot E=0$.
\item[2)] The matrix $S$ defined in \eqref{S} has rank $2$ if $|y ^Hx|<1$, whereas $S=O$ (the zero matrix) if $|y ^Hx|=1$. As a consequence, the equilibria of system \eqref{odefull} for which $|y ^Hx|<1$ (in particular, the minimizers of $|y ^Hx|$) are rank-$2$ matrices. 
\end{enumerate}
\end{theorem}
\proof \textit{1)} The assertion is an immediate consequence of Lemma \ref{lem:2.1}. In particular, the derivative \eqref{decr} is obtained by plugging the right-hand side of \eqref{odefull} in \eqref{dy*x}.

\textit{2)} By \eqref{S} the matrix $S$ has rank not greater than $2$. Moreover, it has rank less than $2$ if and only if $G ^Hx = c y$ or $Gy = cx$ for some $c\in\bb C$. We claim the constant $c$ must vanish in both cases. Indeed, since $Gx=0$ and $y ^HG=0$, we have $\bar cy ^Hx=x ^HGx=0$ or $c y ^Hx = y ^HGy=0$, whence $c=0$ as $y ^Hx\ne 0$. Therefore, the rank of $S$ is less than $2$ if and only if $G ^Hx=0$ or $Gy=0$. As $G$ has rank $n-1$, both conditions are equivalent to have $y=\e^{\i\theta}x$ for some $\theta\in [0,2\pi]$, that is to say $S=O$. Finally, the assertion concerning the equilibria comes from the fact that $S =\Re\lan E,S\ran E$ at a stationary point.
\endproof

\begin{theorem}
\label{thm:2}
Let $x_0$ and $y_0$ be unit-norm right and left eigenvectors, respectively, associated to the simple eigenvalue $\lambda_0$ of the matrix $A$. If $|y_0 ^Hx_0|<1$ then, for $\eps$ small enough, the system \eqref{odefull} has only two stationary points, that correspond to the minimum and the maximum of $|y ^Hx|$ on $\mc S_1$, respectively.
\end{theorem}

\begin{proof} 
As $|y_0 ^Hx_0|<1$, by item \textit{2)} of Theorem~\ref{thm:1} we have that $S=S_0+Q(E,\eps)$ with $S_0$ a non zero constant matrix and 
\be
\label{p:5}
\max_{E\in \mc S_1} \|Q(E,\eps)\|_F = \bigo(\eps)\;.
\ee
The equation for the equilibria reads $F(E,\eps)=O$, where
\[
F(E,\eps) = - S_0 + \Re\lan E,S_0\ran E - Q(E,\eps) + \Re\lan E,Q(E,\eps)\ran E\;.
\]
It is readily seen that $F(E,0)=O$ if and only if $E=E^0_\pm = \pm S_0/\|S_0\|_F$. Moreover, the Jacobian matrix of $F(E,\eps)$ with respect to $E$ at the point $(E^0_\pm,0)$ is given by the linear operator $\mc L_\pm \colon \bb C^{n\times n} \to \bb C^{n\times n}$, such that
\[
\mc L_\pm B = \Re \lan B, S_0 \ran  E^0_\pm + \Re \lan E^0_\pm, S_0\ran B\;, \qquad B\in \bb C^{n\times n}\;.
\]
We shall prove below that $\mc L_\pm$ is invertible. By the Implicit Function Theorem this implies  that there are $\eps_1>0$, $r>0$, and $E^\eps_\pm$ such that $F(E,\eps)=O$ for $\eps\in [0,\eps_1]$ and $\|E-E^0_\pm\|_F<r$ if and only if $E=E^\eps_\pm$. On the other hand, by \eqref{p:5}, if $F(E,\eps)=O$ then $\|E-E^0_+\|=\bigo(\eps)$ or $\|E-E^0_-\|=\bigo(\eps)$. We conclude that $E^\eps_\pm$ are the unique equilibria on the whole hyper-sphere $\mc S_1$ for $\eps$ small enough. Clearly, $E^\eps_+$ [resp.\ $E^\eps_-$] is the maximizer [resp.\ minimizer] of $|y ^Hx|$.

To prove the invertibility of $\mc L_\pm$ we observe that, recalling $E^0_\pm = \pm S_0/\|S_0\|_F$, the equation $\mc L_\pm B=O$ reads, $B = - (\Re\lan B,S_0\ran /\|S_0\|_F^2) S_0$, 
which gives $\Re\lan B,S_0\ran=0$ and therefore $B=O$.
\end{proof}

\section{Projected system of ODEs}
\label{sec:4}

By Theorem \ref{thm:1}, the minimizers of $|y ^Hx|$ onto $\mc S_1$ are rank-$2$ matrices. This suggests to use a rank-$2$ dynamics, obtained as a suitable projection of \eqref{odefull} onto ${ \mc M}_2$. 

\begin{theorem}[The projected system]
\label{thm:3}
Given $\eps\in (0,\eps_0]$, consider the differential system,
\be
\label{eq:r2ode}
\dot E = - P_E\left( S \right) + \Re \lan E, S \ran E\;, \qquad E\in \mc S_1\cap{ \mc M}_2\;,
\ee
where the orthogonal projection $P_E$ is defined in \eqref{p:3}. Then, the right-hand side of \eqref{eq:r2ode} is antiparallel to the projection onto the tangent space $\mc T_E\mc S_1\cap \mc T_E{ \mc M}_2$ of the gradient of $|y ^Hx|$. More precisely,
\be
\label{decr2}
\frac{\d}{\d t} |y ^Hx| = - \eps\, |y ^Hx| \, \|P_E(S)-\Re\lan E,S\ran E\|_F^2\;.
\ee
In particular, $\frac{\d}{\d t} |y ^Hx| =0$ if and only if $\dot E=0$.
\end{theorem}

\begin{proof} 
We remark that the definition is well posed, i.e., $\dot E\in \mc T_E\mc S_1\cap \mc T_E{ \mc M}_2$. Indeed, $\dot E\in \mc T_E{ \mc M}_2$ as $E=P_E(E)$, and
\[
\Re\lan\dot E,E\ran = - \Re \lan E,P_E(S)\ran + \Re\lan E,S\ran = - \Re \lan E,P_E(S)\ran + \Re\lan P_E(E), P_E(S)\ran = 0\;,
\]
as $P_E$ is an orthogonal projection. We next observe that, by \eqref{dy*x}, the steepest descent direction is given by the variational problem,
\[
\mathrm{arg} \min_{\substack{\|D\|_F=1 \\ D\in\mc T_E\mc S_1\cap \mc T_E{ \mc M}_2 }} \Re\big\lan D, S\big\ran\;.
\]
Since $\Re\lan D, S\ran = \Re \lan D, P_E(S)\ran$ for any $D\in \mc T_E{ \mc M}_2$, the solution to this problem is
\[
D = - \frac{P_E(S)-\Re\lan E,S\ran E}{\|P_E(S)-\Re\lan E,S\ran E\|_F}\;.
\]
This proves the first assertion, while the derivative \eqref{decr2} is obtained by plugging the right-hand side of \eqref{eq:r2ode} in \eqref{dy*x} and using $\frac{\d}{\d t} |y ^Hx|  = \frac 12  |y ^Hx| ^{-1}  \frac{\d}{\d t} |y ^Hx|^2$.
\end{proof}

\begin{remark}\rm
\label{rem:1}
It is worthwhile to notice that along the solutions to \eqref{odefull} or \eqref{eq:r2ode} we have $\frac{\d}{\d t} (y ^Hx) = - \eps\, y ^Hx \, \|\dot E\|_F^2$, so that the condition $\Im(y ^Hx)=0$ is preserved by the dynamics. 
\end{remark}

In the sequel, it will be useful the following rewriting of the projected system, in terms of the representations of  ${ \mc M}_2$ and $\mc T_E { \mc M}_2$ discussed at the end of Section \ref{sec:2}. By introducing the notation,
\be
\label{eq:pqrs}
p = U ^H y\;, \quad q = V ^H x\;, \quad r = U ^H G ^H x\;, \quad s = V ^H G y\;,
\ee
we can write \eqref{eq:r2ode} as
\be
\label{ode-utv} 
\left\{\begin{array}{l}
\dot T = {}-\left( p s ^H + r q ^H \right) + \left( s ^H T ^H p + q ^H T ^H r \right) T \;,\\ 
\dot U = {}-\big( (y - U p) s ^H + (G ^H x - U r) q ^H \big) T^{-1}\;, \\
\dot V = {}-\big( (G y - V s) p ^H + (x - V q) r ^H \big) T^{-H} \;.
\end{array}\right.
\ee

\subsection{Stationary points of the projected ODEs}
\label{sec:4.1}

We start by providing a characterizing result for stationary points of system \eqref{eq:r2ode}.

\begin{lemma}
\label{lem:stat}
Given $\eps\in (0,\eps_0]$, assume $E\in\mc S_1\cap{ \mc M}_2$ is a stationary point of \eqref{eq:r2ode} (or equivalently of \eqref{ode-utv}) such that $\lambda$ is not an eigenvalue of $A$ and $|y ^H x| < 1$. Then $P_E(S) \ne 0$.
\end{lemma}

\begin{proof} 
Assume by contradiction that $P_E(S)=0$; then we would get
\[
S = \left( I - U U ^H\right) S \left( I - V V ^H \right)\;.
\]
The previous implies
\begin{eqnarray*}
U ^H S & = & 0 \qquad \Longrightarrow \qquad U ^H S y = r (x ^Hy) = 0 \;,  \\
S V & = & 0 \qquad \Longrightarrow \qquad V ^H S ^H x = s (y ^H x) = 0\;,
\end{eqnarray*}
whence, as $|y ^Hx|\ne 0$ onto $\mc S_1\cap{ \mc M}_2$, $r=0$, $s=0$. Inserting these formul\ae \ into \eqref{ode-utv} we would obtain
\be
\left\{\begin{array}{l}
\dot T = 0\;,\\ 
\dot U = {}- \left( G ^H x x ^H V \right) T^{-1}\;, \\
\dot V = {}- \left( G\, y y ^H U  \right) T^{-H} \;.
\end{array}\right.
\ee
In order that $\dot E = 0$ it has to hold necessarily $\dot U = 0$ and $\dot V = 0$.  Since $ |y ^H x| < 1$ implies $G ^Hx \neq 0$ and $G y \neq 0$, the previous relations imply $V ^H x = 0$ and $y ^H U = 0$, so that $E x = 0$ and $y ^H E = 0$. As a consequence, $\lambda$ would be an eigenvalue of $A$ which contradicts the assumptions.  This means that $P_E(S) \neq O$.
\end{proof}

At a stationary point the equation $P_E(S)=\Re\lan E,S\ran E$  reads,
\be
\label{eq:r2sp}
E = \mu P_E(S)\;,
\ee
for some nonzero $\mu \in \bb R$. In this case, as 
\[
\Re \left( y ^H E G y + x ^H G E x \right) = \Re \lan S, E \ran  = \Re \lan \mu P_E(S),  P_E(S) \ran = \frac{1}{\mu} \lan  E,E\ran \ne 0 \;,
\]
assuming $E=U T V ^H$, we get
\be
\label{eq:nz}
\Re \left( y ^H U T V ^H G y + x ^H G U T V ^H x \right) \ne  0\;.
\ee

Recalling \eqref{p:3}, we are interested in studying
\be
\label{eq:B}
B = \left( I - U U ^H\right) S \left( I - V V ^H \right) = 
\left( I - U U ^H\right) \left( y y ^H G ^H + G ^H x x ^H \right) \left( I - V V ^H \right)\;.
\ee

\begin{theorem}
\label{th:eprops}
Given $\eps\in (0,\eps_0]$, assume $E=U T V ^H\in\mc S_1\cap{ \mc M}_2$ is a stationary point of \eqref{eq:r2ode} (or equivalently of \eqref{ode-utv}) such that $\lambda$ is not an eigenvalue of $A$ and $|y ^H x| < 1$. Then it holds $E = \mu S$ for some real $\mu$.
\end{theorem}

\begin{proof}
In order to prove the theorem we show that $B=O$. From the nonsingularity of $T$ we get at a stationary point, 
\begin{eqnarray}
(y - U p) s ^H + (G ^H x - U r) q ^H & = & 0\;,
\label{eq:c1}
\\[0.2cm]
(G y - V s) p ^H + (x - V q) r ^H & = & 0\;.
\label{eq:c2}
\end{eqnarray}
The assumption that $\lambda$ is not an eigenvalue of $A$ implies $E x \ne 0$ and $y ^H E \ne 0$, that means $p \neq 0$ and $q \neq 0$. Moreover, by Lemma~\ref{lem:stat} the condition \eqref{eq:nz} is satisfied in our case. In order that \eqref{eq:c1} and \eqref{eq:c2} are fulfilled we have several possible cases.

Consider \eqref{eq:c1}. The following are the possible cases.   
\begin{itemize}
\item[(1-i) ] $s=0$, $G ^H x - U r = 0$. 
\item[(1-ii) ] $y = U p$, $G ^H x - U r = 0$. This would imply 
\[
y = U U ^H y \quad \mbox{and} \quad G ^H x = U U ^H G ^H x \qquad \Longrightarrow \qquad
S = U U ^HS\;,
\]
and thus $B=O$.
\item[(1-iii) ] $s \propto q$. This would imply $V ^H G y \propto V ^H x$.  
\end{itemize}

Now consider (\ref{eq:c2}). The following are the possible cases.   
\begin{itemize}
\item[(2-i) ] $r=0$, $G y - V s = 0$. 
\item[(2-ii) ] $x = V q$, $G y - V s = 0$. This would imply 
\[
x = V V ^H x \quad \mbox{and} \quad G y = V V ^H G y \qquad \Longrightarrow \qquad
S = S V V ^H\;,
\]
and thus $B=O$.
\item[(2-iii) ] $p \propto r$. This would imply $U ^H G ^H x \propto U ^H y$.
\end{itemize}

Assume that $B \neq O$, which excludes (1-ii) and (2-ii).
Assume (1-i) and (2-i) hold. This would imply
\[
s = V ^H G y = 0 \qquad \mbox{and} \qquad r = U ^H G ^H x = 0\;,
\]
which would contradict (\ref{eq:nz}) .

Assume (1-iii) and (2-iii) hold. This would imply $p s ^H + r q ^H \propto U ^H y x ^H V$,
and from the first of (\ref{ode-utv}), that $T$ has rank-$1$ which contradicts the invertibility of $T$. The same conclusion holds if (1-i) and (2-iii) hold, since $s=0$ would also imply that $T \propto r q ^H$ has rank-$1$ and if (1-iii) and (2-i) hold, because in this case we would have $r=0$ and $T \propto p s ^H$
still of rank-$1$.
\end{proof}

To summarize, if $\lambda$ is not an eigenvalue of $A$ the stationary points of the projected and the unprojected ODEs coincide (recall that if $|y ^Hx|=1$ this is obvious as  $S=O$ in this case, see Theorem~\ref{thm:1}). Moreover, since $E$ is proportional to $S$ at such points, by the same arguments leading to Theorem \ref{thm:2} we obtain the following result.

\begin{corollary}
\label{cor:2}
Under the same assumptions of Theorem {\rm\ref{th:eprops}}, for sufficiently small $\eps$ the projected ODE \eqref{eq:r2ode} has only two stationary points.
\end{corollary}

\section{The real structured case}
\label{sec:5}

We now assume that the matrix $A$ is real and we restrict the perturbations $E$ to be real as well. To our knowledge there are no methods to compute the most ill conditioned eigenvalue in the real $\eps$-pseudospectrum, 
\be
\label{eq:realps}
\Lameps^{\R}(A) = \{ \lambda \in \C : \lambda \in \Lambda\left(A+E\right) \ \mbox{for some} \  E \in \R^{n\times n} \ \mbox{with} \ \|E\|_F \le \eps\}.
\ee 
that is the eigenvalue of $A + \eps E$ to which corresponds $\min y ^Hx$.

We denote by $\mc R_1$ the unitary hyper-sphere in $\bb R^{n\times n}$ and fix $\widetilde \eps_0>0$ such that for any $\eps\in [0,\widetilde\eps_0]$ and $E\in \mc R_1$ the matrix $A+\eps E$ has a simple eigenvalue $\lambda$ close to $\lambda_0$. The same reasoning as in the proof of Lemma \ref{lem:2.1} gives, for a smooth path $E(t)\in \mc R_1$,
\be
\label{xyr}
\frac{\d}{\d t} |y ^Hx| = \eps\,|y ^Hx| \, \big\lan \dot E,\Re(S)\big\ran\;,
\ee
from which the steepest descent direction is given by the variational problem,
\[
\mathrm{arg} \min_{\substack{D\in\mc R_1 \\ \lan D,E\ran=0}} = - \mu\big(\Re(S)-\lan E,\Re(S)\ran E\big)\;,
\]
where $\mu$ is the normalization constant. Note that the matrix $\Re (S)$ has rank not greater than $4$. Clearly, for a real eigenvalue $S$ is real and the situation is identical to the one considered in the unstructured case $\K=\C$;
so the peculiar difference arises when we consider non-real eigenvalues.

Let ${ \mc M}_4$ be the manifold of the real matrices of rank-$4$. The matrix representations both in ${ \mc M}_4$ and in the tangent space $T_E{ \mc M}_4$ are analogous to \eqref{p:1}, \eqref{p:2}, provided that $U,V \in \R^{n\times 4}$ have orthonormal columns and $T \in \R^{4\times4}$ is nonsingular, see \cite[Sect.~2.1]{KL07}. More precisely, any rank-$4$ matrix of order $n$ can be written in the form,
\be
\label{factors}
E= U T V^{\rm T}\;,
\ee
with, now, $U$ and $V$ such that $U^{\rm T}U=I_4$ and $V^{\rm T}V=I_4$, where $I_4$ is the identity matrix of order 4, and $T \in \R^{4\times 4}$ is nonsingular. As before, since this decomposition is not unique, we use a unique decomposition on the tangent space. For a given choice of $U,V,T$ any matrix $\delta E\in T_{E}\mathcal{M}_4$ can be uniquely written as 
\[
\delta E= \delta U T V^{\rm T} + U \delta T V^{\rm T} + U T \delta V^{\rm T}\;,
\]
with $U^{\rm T}\delta U=0$, $V^{\rm T}\delta V=0$. Accordingly, the orthogonal projection of a matrix $Z\in \bb R^{n\times n}$ onto the tangent space $T_E{ \mc M}_4$  is defined by 
\be
\label{projeq}
\widetilde P_E(Z)=Z-P_U Z P_V\;,
\ee
where $P_U=(I-U U^\tp)$ and $P_V=(I-V V^{\rm T})$.

\section{System of ODEs in the real case}
\label{sec:6}

Given $\eps\in (0,\widetilde \eps_0]$ the role of the differential system in \eqref{odefull} is now played by
\be
\label{realodefull}
\dot E = - \Re (S) + \lan E,\Re (S)\ran E\;, \qquad E\in \mc R_1\;.
\ee
More precisely, the right-hand side of \eqref{realodefull} is antiparallel to the projection onto the tangent space $\mc T_E\mc R_1$ of the gradient of $|y ^Hx|$ and
\be
\label{decreal}
\frac{\d}{\d t} |y ^Hx| = - \eps\, |y ^Hx| \, \|\Re(S)-\lan E,\Re(S)\ran E\|_F^2\;.
\ee

The proof of Theorem \ref{thm:2} is easily adapted to the real case. Therefore, under the same hypothesis, the system \eqref{realodefull} has only two stationary points, that correspond to the minimum and the maximum of $|y ^Hx|$ on $\mc R_1$, respectively.

A natural question is concerned with the possibility of $\Re(S)$ to vanish.
We have the following result concerning the matrix $\Re(S)$.

\begin{theorem}
\label{th:ReS}
Assume that the matrix $B$ is real and has a pair of simple complex conjugate eigenvalues
$\lambda$ and $\bar\lambda$. Let $y$ and $x$ be its left and right eigenvectors
associated to $\lambda$ such that $y ^H x < 1$. Let $G$ be the G-inverse of $B - \lambda I$ and $S$ be given by \eqref{S}. Then $\Re(S)$ is different from zero.
\end{theorem}

\begin{proof}
First observe that the eigenvectors $x$ and $y$ are necessarily genuinely complex vectors, that is $\Re(x) \ne 0$, $\Im(x) \ne 0$, $\Re(y) \ne 0$, and $\Im(y) \ne 0$. Let us denote the range of a matrix $M$ as $\cR(M)$. By definition of $S$ (see (\ref{S})) we have $
\cR(S) = {\rm span}\left( y, G ^H x \right)$.

We prove the result by contradiction. Assume that $S$ is purely imaginary, that is $
\Re(S) = {\rm O}$. Under this assumption, recalling that $S$ is a rank-$2$ matrix, we have that $v \in \cR(S)$ implies $\bar{v} \in \cR(S)$, and therefore
$$
\cR(S) = {\rm span} \left( y, \conj{y}, G ^Hx, \conj{G ^Hx} \right)
$$
has dimension $2$. Being $y$ and $G ^Hx$ linearly independent, we get $\conj{y} =\alpha y + \beta G ^H x$. A left premultiplication by $x ^H$ gives
\begin{eqnarray}
x ^H\conj{y} & = & \alpha x ^Hy + \beta x ^HG ^H x \qquad \Longrightarrow \quad \alpha = 0\;,
\label{eq:conjy2}
\nonumber
\end{eqnarray}
which is due to the fact that (i) \ $x ^H\conj{y} = 0$, by well-known bi-orthogonality of left and right eigenvectors, (ii) \ $G x = 0$, a property of the group inverse $G$, and (iii) \ 
$x ^H y \neq 0$ by simplicity of $\lambda$.
This implies $\conj{y} \propto  G ^Hx $. In a specular way, we have that 
$$
\cR(S^{\rm T}) = {\rm span} \left( x, \conj{x}, G y, \conj{G y} \right)
$$
has dimension $2$. Proceeding in the same way we obtain that $\bar{x} \propto G y$.

Now recall that (see \cite{MS88}) a vector $v$ is an eigenvector of $B-\lambda I$ corresponding 
to the eigenvalue $\mu$ if and only if $v$ is an eigenvector of $G$ corresponding to the eigenvalue  
$\mu^\#$, where $\mu^\# = 1/\mu$ if $\mu \neq 0$ and $\mu^\#=0$ if $\mu = 0$. 
To recap we have
\begin{eqnarray}
&& G y = \gamma \conj{x}\;, \qquad G \conj{x} = \iu \frac{1}{2 \Im(\lambda)} \conj{x}\;,
\label{eq:d1}
\\ 
&& G ^H x = \eta \conj{y}\;, \qquad G ^H \conj{y} = -\iu \frac{1}{2 \Im (\lambda)} \conj{y}\;,  
\label{eq:d2}
\end{eqnarray}
with $\gamma \neq 0$ and $\eta \neq 0$.

Since $x \in \ker(G)$ and $y \in \ker(G ^H)$ are the only vectors in the kernels of $G$ and $G ^H$,
respectively, we deduce by (\ref{eq:d1}) and (\ref{eq:d2}),
\begin{eqnarray*}
&& x \propto \frac{1}{\gamma} y + 2\,\iu\,\Im(\lambda) \conj{x}\;, \qquad 
y \propto \frac{1}{\eta} x - 2\,\iu\,\Im(\lambda) \conj{y}\;,  
\end{eqnarray*}
which imply ${\rm span} \left( x, \conj{x} \right) = {\rm span} \left( y, \conj{y} \right)$. The previous implies, by bi-orthogonality of left and right eigenvectors, that
the set $X=\{ x,\conj{x}\}$ is orthogonal to the set $\widehat{X}$ consisting of the remaining 
$n-2$ right eigenvectors of $B$ and similarly the set $Y=\{y, \conj{y}\}$ is orthogonal to the set
$\widehat{Y}$ of the remaining $n-2$ left eigenvectors of $B$. Note that ${\rm span}(X)$ and 
${\rm span}(\widehat{X})$ are right-invariant subspaces of both $B$ and $B^{\rm T}$.

Denote by ${\rm Orth}(C)$ a real orthonormal basis for the range of $C$ and 
define $U = {\rm Orth}(X) \in \R^{n,2}$ (as determined by the procedure to get
the Schur canonical form of $B$) and $V={\rm Orth}(\widehat{X}) \in \R^{n,n-2}$. 
Set $Q=\left(U, V \right) \in \R^{n,n}$, which implies $Q$ is an orthogonal matrix. Now 
consider the similarity transformation associated to $Q$,
\[
\widetilde{B} = Q^{\rm T} B Q =
\left( 
\begin{array}{cc}
B_1 & {\bf O}^{\rm T} \\
{\bf O} & B_2  
\end{array}
\right) \;, 
\]
where ${\bf O}$ stands for the $(n-2) \times 2$-dimensional zero matrix, $B_1 \in \R^{2,2}$ and $B_2 \in \R^{n-2,n-2}$. This means that $\widetilde{B}$ is block-diagonal and the matrix 
\begin{equation}
B_1 = \begin{pmatrix} 
\varrho    & \sigma \\
{}-\tau  & \varrho  
\end{pmatrix}
\label{eq:B1}
\end{equation}
is such that $\varrho = \Re(\lambda)$ and $\sigma > 0, \tau > 0$ with 
$\sigma \tau = \Im(\lambda)^2 > 0$ so that $B_1$ has eigenvalues $\lambda$ 
and $\conj{\lambda}$.
If $\sigma=\tau$ then $B_1$ is normal, which implies that the pair of right 
and left eigenvectors associated to $\lambda$, say $\widetilde{x}, \widetilde{y}$
(scaled to have unit $2$-norm and real and positive Hermitian scalar product)
is such that $\widetilde{y} ^H \widetilde{x} = 1$. 
Since $x = Q \widetilde{x}$ and $y = Q \widetilde{y}$, the orthogonality of $Q$ implies $y ^Hx=1$,
which gives a contradiction. As a consequence we can assume $\tau \ne \sigma$.

By the properties of the G-inverse we have that
\[
\widetilde{G} = Q^{\rm T} G Q =
\left( 
\begin{array}{cc}
G_1 & {\bf O}^{\rm T} \\
{\bf O} & G_2  
\end{array}
\right)\,,
\]
where $G_1$ is the group inverse of $B_1-\lambda I$ and $G_2$ is the inverse of $B_2-\lambda I$, which is nonsingular.
It is direct to verify the following formula for the G-inverse, by simply checking the three conditions in Definition \ref{def:groupinv},
\[
G_1 =
\left(
\begin{array}{rr}
 \frac{\iu}{4 \sqrt{\sigma \tau}} & -\frac{1}{4 \tau} \\
 \frac{1}{4 \sigma} & \frac{\iu}{4 \sqrt{\sigma \tau}}
\end{array}
\right)\,.
\]
It follows that also $Q^{\rm T} S Q$ is block triangular so that we write
\begin{eqnarray}
\widetilde{S} = Q^{\rm T} S Q =
\left( 
\begin{array}{cc}
S_1 & {\bf O}^{\rm T} \\
{\bf O} & S_2  
\end{array}
\right)\,,
\label{eq:QtSQ}
\end{eqnarray}
with
\begin{eqnarray}
S_1 & = & \widetilde{y}_1 \widetilde{y}_1 ^H G_1 ^H + G_1 ^H  \widetilde{x}_1 \widetilde{x}_1 ^H\;,
\label{eq:S1}
\end{eqnarray}
where $\widetilde{y}_1 \in \C^2$ and $\widetilde{x}_1 \in \C^2$ are the projections on 
${\rm span}(e_1,e_2)$ (the subspace spanned by the first two vectors of the canonical basis) 
of the eigenvectors of $\widetilde{B}$ associated to $\lambda$, that is 
$\widetilde{x} = Q^{\rm T} x$ and $\widetilde{y} = Q^{\rm T} y$, 
\[
\widetilde{x} = \nu_x^{-1} \left( \begin{array}{ccccc}
{} \iu \frac{\sqrt{\sigma}}{\sqrt{\tau}} & 1 & 0 & \ldots & 0 \end{array} \right)^{\rm T},
\qquad \widetilde{y} = \nu_y^{-1} \left( \begin{array}{ccccc}
{}-\iu \frac{\sqrt{\tau}}{\sqrt{\sigma}} & 1 & 0 & \ldots & 0 \end{array} \right)^{\rm T},
\]
where $\nu_x = \sqrt{\frac{\sigma}{\tau}+1}$ and $\nu_y = \sqrt{\frac{\tau}{\sigma}+1}$ are such that 
$\| \widetilde{x} \| = \| \widetilde{y} \| = 1$ and $\widetilde{x} ^H \widetilde{y} \in \R^+$.
Finally, we obtain
\begin{eqnarray}
S_1 & = & \left(
\begin{array}{ll}
 0 & \frac{\tau-\sigma}{2 \sigma \left(\sigma+\tau \right)} \\
 \frac{\tau-\sigma}{2 \tau \left(\sigma+\tau\right)} & 0
\end{array}
\right)\,,
\nonumber
\end{eqnarray}
which is real and cannot vanish due to the fact that $\sigma \neq \tau$.

Recalling that $Q$ is real, if $S$ were purely imaginary then $S_1$ would be purely imaginary as well, which gives a contradiction. 

We remark that it can also be shown that $S_2 = O$ in (\ref{eq:QtSQ}).
\end{proof}

\begin{remark}\rm
Note that when $A$ is real and we compute $S$ for a complex eigenvalue, it can
occur that $S$ is real. The simplest example is given by the matrix {\rm (\ref{eq:B1})}.
\end{remark}

According to Theorem \ref{th:ReS} we have that $\Re(S) \neq 0$ for every path of
genuinely complex eigenvalues. Based on this result we can characterize stationary
points of (\ref{realodefull}) to have rank (at most) $4$.

This suggests to project the ODE on the rank-$4$ manifold of real matrices. 

\subsection{Projected system of ODEs in the real case}
\label{subsec:6.1}

The following theorem characterizes the projected system onto the tangent space $\mc T_E {\mc M}_4$.

\begin{theorem}[The real projected system]
\label{thm:3.1}
Given $\eps\in (0,\widetilde\eps_0]$, consider the differential system,
\be
\label{realeq:r2ode}
\dot E = - \widetilde P_E\left( \Re(S) \right) + \lan E, \Re(S) \ran E\;, \qquad E\in \mc R_1\cap{ \mc M}_4\;,
\ee
where the orthogonal projection $\widetilde P_E$ is defined in \eqref{projeq}. Then, the right-hand side of \eqref{realeq:r2ode} is antiparallel to the projection onto the tangent space $\mc T_E\mc R_1\cap \mc T_E{ \mc M}_4$ of the gradient of $|y ^Hx|$. More precisely,
\be
\label{decr2.1}
\frac{\d}{\d t} |y ^Hx| = - \eps\, |y ^Hx| \, \|\widetilde P_E(S)-\lan E,\Re(S)\ran E\|_F^2\;.
\ee
\end{theorem}

\begin{proof} 
By \eqref{xyr}, the steepest descent direction is given by the variational problem,
\[
\mathrm{arg} \min_{\substack{\|D\|_F=1 \\ D\in\mc T_E\mc R_1\cap \mc T_E{ \mc M}_4}}\big\lan D, \Re(S)\big\ran\;.
\]
Since $\lan D, \Re(S)\ran = \lan D, \widetilde P_E(\Re(S)) \ran$ for any $D\in \mc T_E{ \mc M}_4$, this problem has solution,
\[
D = - \frac{\widetilde P_E(\Re(S))-\lan E,\Re(S)\ran E}{\|\widetilde P_E(\Re(S))-\lan E,\Re(S)\ran E\|_F}\;.
\]
\end{proof}

Based on the previous theorem, our aim is that of writing a system of differential equations on the manifold of rank-$4$ matrices for the projected system. Given $\eps\in (0,\widetilde\eps_0]$, consider the differential system (\ref{realeq:r2ode}). To obtain the differential equation in a form that uses the factors in the decomposition \eqref{factors} rather than the full $n\times n$ matrix $E$, we use the following result.

\begin{lemma}
\cite[Prop.~2.1]{KL07}
\label{lem:usv}  
For $E=UTV^\tp \in \M_4$ with nonsingular $T\in\R^{4\times 4}$ and with $U\in\R^{n\times 4}$ and $V\in\R^{n\times 4}$ having orthonormal columns, the equation $\dot E=\widetilde P_E(Z)$ is equivalent to 
$
\dot E = \dot U T V^\tp  + U \dot T V^\tp  + U T\dot V^\tp , 
$
where
\be
\label{Rodes} 
\left\{\begin{array}{l}
\dot T = U^\tp  Z V\;, \\[1mm] 
\dot U = (I-UU^\tp) Z V T^{-1}\;,\\[1mm] 
\dot V = (I-VV^\tp) Z^\tp  U T^{-\tp}\;.
\end{array}\right.
\ee
\end{lemma}
In order to explicit (\ref{Rodes}) we write
\be 
\Re(S)= \Re(y y ^HG ^H + G ^H x x ^H)= Y W^{\rm T} + Z X^{\rm T}\;,
\ee
where 
\be
\left\{
\begin{array}{lll} 
X=(\Re(x), \Im(x))\;, &  \quad & Y=(\Re(y), \Im(y)) \;, \\[1mm]
W=(\Re(G y), \Im(G y))\;, & \quad & Z=(\Re(G ^Hx), \Im(G ^Hx))\;,
\end{array} 
\right.
\ee
are matrices in $\R^{n\times 2}$.

By simple algebraic manipulations we obtain the following system of ODEs,
\be
\label{ode-utv4} 
\left\{\begin{array}{l}
\dot T = {}-\left( P \S^\tp + R Q^\tp \right)
+ {\rm trace} \left( \S^\tp T^\tp P + Q^\tp T^\tp R \right) T \;,\\[1mm] 
\dot U = {}-\bigl( (Y - U P) \S^\tp + (Z - U R) Q^\tp \bigr) T^{-1}  \;, \\[1mm]
\dot V = {}-\bigl(  (W - V \S) P^\tp + (X - V Q) R^\tp  \bigr) T^{-\tp} \;,
\end{array}\right.
\ee
where
\be
\label{eq:Rpqrs}
P = U^{\rm T} Y\;, \quad
Q = V^{\rm T} X\;, \quad
R = U^{\rm T} Z\;, \quad
\S = V^{\rm T} W\;,
\ee
are matrices in $\R^{4\times 2}$. 

We provide now a characterizing result for stationary points of \eqref{realeq:r2ode}.

\begin{lemma}\label{lem:statr}
Given $\eps\in (0,\widetilde \eps_0]$, assume $E\in\mc R_1\cap\mc M_4$ is a stationary point of \eqref{realeq:r2ode}  such that $\lambda$ is not an eigenvalue of $A$ and $|x ^H y| < 1$. Then $\widetilde P_E(\Re(S)) \neq 0$.
\end{lemma}

\begin{proof}
 Assume by contradiction that $\widetilde P_E(S)=0$; then we have,
\be\label{prores}
\Re(S) = \left( I - U U^{\rm T}\right) \Re(S) \left( I - V V^{\rm T} \right)\;.
\ee
With this notation equation \eqref{prores} becomes,
\be
Y W^{\rm T}+ Z X^{\rm T}  = (I-UU^{\rm T})(Y W^{\rm T}+Z X^{\rm T})(I-VV^{\rm T})\;.
\ee
Multiplying by $U^{\rm T}$ to the left and $V$ to the right we obtain $U^{\rm T} Y W^{\rm T} V + U^{\rm T} Z X^{\rm T} V = 0$. By Theorem \ref{th:ReS}  we have that $Y W^{\rm T} +Z X^{\rm T}$ cannot be zero, therefore we have 
the following possibilities, 
\begin{eqnarray} 
\label{p1} 
&& U^{\rm T} Y=0 \quad \mbox{ and }\quad  U^{\rm T} Z=0\;,  \\
\label{p2} 
&& U^{\rm T} Y=0 \quad \mbox{ and }\quad   X^{\rm T} V=0 \;, \\
\label{p3} 
&& W^{\rm T} V=0 \quad \mbox{ and }\quad   X^{\rm T} V=0\;, \\
\label{p4} 
&& W^{\rm T} V=0 \quad \mbox{ and }\quad   U^{\rm T} Z=0\;.
\end{eqnarray}
Each of the conditions \eqref{p1}, \eqref{p2}, and \eqref{p3} imply that $\lambda$ is an eigenvalue of $A$ and this contradicts the hypothesis. Instead, by \eqref{eq:Rpqrs}, condition \eqref{p4} means $\S=0$ and $R=0$; whence, replacing these values in \eqref{ode-utv4}, 
we get,
\be
\left\{\begin{array}{l}
\dot T = 0\;,\\ 
\dot U = {}- Z Q^\tp T^{-1} \;, \\
\dot V = {}-W P^\tp T^{-T} \;.
\end{array}\right.
\ee
In order to have $\dot E = 0$  we need $\dot U = 0$ and $\dot V = 0$.  
Since $ |x ^H y| < 1$ implies $G ^Hx \neq 0$ and $G y \neq 0$, we have
$Z \neq 0$ and $W \neq 0$. Moreover, by a rotation $(x,y) \to (\mathrm{e}^{\mathrm{i}\phi}x,\mathrm{e}^{\mathrm{i}\phi}y)$ we can assume both $G ^Hx$ and $Gy$ genuine complex vectors. 
As a consequence, the previous relations imply $Q=V^\tp X = 0$ and $P^\tp=Y^\tp U=0$, so that $E x = 0$ and $y ^H E = 0$. Therefore $\lambda$ would be an eigenvalue of $A$, which contradicts the assumptions.  This means that $\widetilde P_E(\Re(S)) \neq O$. 
\end{proof}

Consequently, at a stationary point we have $E=\mu \widetilde P_E(\Re(S))$ for some real $ \mu\neq 0$. Therefore,
\begin{eqnarray*} 
\Re(x ^HG E x + y ^H EG ^Hy) &=& \langle \Re(S), E\rangle= \langle E, \Re(S) \rangle\\
			      &=& \langle \mu \widetilde P_E(\Re(S)), \widetilde P_E(\Re(S))\rangle
			      = \frac{1}{\mu} \langle E,E\rangle \neq 0\;,
\end{eqnarray*}						
that means it never vanishes.

\subsection{Stationary points of the projected system of ODEs in the real case}
\label{subsec:6.2}

In order to study stationary points of (\ref{ode-utv4}) we define,
\be
\label{eq:BR}
B = \left( I - U U^\tp\right) \Re(S) \left( I - V V^\tp \right) = 
\left( I - U U^\tp \right) \left( Y W^\tp + Z X^\tp  \right) \left( I - V V^\tp \right)\;.
\ee

\begin{theorem}
Given $\eps\in (0,\widetilde\eps_0)$, assume $E=U T V ^\tp\in\mc R_1\cap{ \mc M}_4$ is a stationary point of \eqref{realeq:r2ode} (or equivalently of \eqref{ode-utv4}) such that $\lambda$ is not an eigenvalue of $A$ and $|y^H x| < 1$. Then it holds $E = \mu \Re(S)$ for some real $\mu$.
\label{th:rstat}
\end{theorem}

\begin{proof}
To prove that $\widetilde P_E(\Re(S)) = \Re(S)$ we have to show that the matrix $B$ in (\ref{eq:BR}) is zero. Assume $(T,U,V)$ is a stationary point of \eqref{ode-utv4}. The first equation yields
\begin{eqnarray}
P N^\tp + R Q^\tp & = & c\, T \qquad \Longrightarrow \qquad
U^\tp \Re(S) V = c\, T\;,
\label{eq:Tstat}
\end{eqnarray}
where $c$ is a nonzero constant. 

By the assumptions we have that $Q \neq 0$ and $P \neq 0$, otherwise we would have
either $E x = 0$ or $y^H E = 0$, which would imply that $\lambda$ is an eigenvalue of $A$.
To fulfil the second equation in \eqref{ode-utv4}, we obtain the following possibilities,
\begin{itemize}
\item[(1-i) ]   $N = 0$ and $Z - U R = 0$; this would imply 
\begin{equation}
Z = U U^\tp Z\;.
\label{eq:Zstat}
\end{equation} 
\item[(1-ii) ]  $Y - U P = 0$ and $Z - U R = 0$; this would imply (\ref{eq:Zstat}) and
\begin{equation}
Y = U U^\tp Y\;.
\label{eq:Ystat}
\end{equation} 
\item[(1-iii) ] Both term in the second of \eqref{ode-utv4} do not vanish, which
                implies ${\rm span}(N) = {\rm span}(Q)$; from (\ref{eq:Tstat}) we 
								gather that $T$ has rank-$2$.
\end{itemize}
Similarly, to satisfy the third equation in \eqref{ode-utv4}, we obtain the following
cases,
\begin{itemize}
\item[(2-i) ]   $R=0$ and $W-V N = 0$; this would imply 
\begin{equation}
W = V V^\tp W\;.
\label{eq:Wstat}
\end{equation} 
\item[(2-ii) ]  $V - W N = 0$ and $X - V Q = 0$; this would imply (\ref{eq:Zstat}) and
\begin{equation}
X = V V^\tp X\;.
\label{eq:Xstat}
\end{equation}
\item[(2-iii) ] Both term in the second of \eqref{ode-utv4} do not vanish, which
                implies ${\rm span}(P) = {\rm span}(R)$; from (\ref{eq:Tstat}) we 
								gather that $T$ has rank-$2$.
\end{itemize}
Cases (1-iii) and (2-iii) would imply that $T$ is singular, which can be excluded. 
Cases (1-ii) and (2-ii) imply that $B=0$ (by replacing respectively $Y$ and $Z$
with (\ref{eq:Ystat}) and (\ref{eq:Zstat}) into (\ref{eq:BR}) and $W$ and $X$
with (\ref{eq:Wstat}) and (\ref{eq:Xstat}) into (\ref{eq:BR})).
Similarly assume that (1-i) and (2-i) hold true simultaneously. Again this would imply,
by replacing (\ref{eq:Zstat}) and (\ref{eq:Wstat}) into (\ref{eq:BR}) that $B=0$.
This proves that $\widetilde P_E(\Re(S)) = \Re(S)$. 
\end{proof}

An immediate consequence of Theorem \ref{th:rstat} is the following.

\begin{corollary}
The rank-$4$ stationary points of {\rm (\ref{realodefull})} and {\rm (\ref{ode-utv4})} coincide.
\end{corollary}

\section{Computing the distance to defectivity}
\label{sec:7}

For a given $\eps$ we define, 
\begin{eqnarray}
r(\eps) & = & \min_{E : \| E \|_F = 1} y ^Hx, \quad
\mbox{where} \quad x \ \mbox{and} \ y \quad \mbox{right and left eigenvectors of}
\ A + \eps E 
\nonumber
\\
& & \mbox{with} \ y ^Hx \ge 0\;.
\label{eq:reps}
\end{eqnarray}
Given a matrix $A$ with all distinct eigenvalues and $\delta \ge 0$ we look for
$$
\eps^{\delta,*} = \arg\min\limits_{\eps > 0} \{ \eps: \ r(\eps) \le \delta \}\;. 
$$
Starting from $\eps > 0$ such that $r(\eps) > \delta$, which can be computed by integrating the ODEs  in previous sections, we want to compute a root $\eps^{\delta,*}$ of the equation
$
r(\eps) = \delta.
$
We expect generically that $r(\eps)$ is not smooth at zero, when two eigenvalues coalesce to form a Jordan block. Nevertheless, the eigenvalue $\lambda(\eps)$ and the eigenvectors $x(\eps)$ and $y(\eps)$ are smooth if $\eps<\eps^{0,*}$ so that $r(\eps)$ is smooth as well. 

More precisely, by Theorem \ref{thm:2} and Corollary \ref{cor:2} we know that at least for $\eps$ small enough there exist a unique branch of minimizers $E(\eps)$, which are non degenerate solutions to the stationary equation $P_E(S) = \Re\lan E,S\ran E$ and $\widetilde P_E\left( \Re(S) \right) = \lan E, \Re(S) \ran E$, respectively. These equations have the form $F(E,\eps)  = 0$, with $F$ a (complex or real) analytic function of $E$ and $\eps$, hence the solution $E(\eps)$, as well as the corresponding simple eigenvalue $\lambda(\eps)$ and eigenvectors $x(\eps)$ and $y(\eps)$, are (complex or real) analytic functions of $\eps$. In order to derive an equation for $\eps$ to approximate $\eps^{\delta,*}$, the $\delta$-distance to defectivity, we need to compute the derivative of $r(\eps)$ with respect to $\eps$.

\begin{theorem}
\label{th:der}
Let $r(\eps)$, $\eps \in [\underline{\eps},\overline{\eps}]$, be a branch of minima such that $r(\eps) > 0$ for all $\eps$, so that $\lambda(\eps)$ is a simple eigenvalue of
$
A + \eps E (\eps),
$
where $E(\eps)$ is a smooth minimizer of {\rm (\ref{eq:reps})} either with $E \in \bb C^{n\times n}$ or $E \in \bb R^{n\times n}$, that is $\lambda(\eps) \in \Lambda \left( A + \eps E(\eps) \right)$ with $\| E(\eps) \|_F = 1$ for all $\eps$.
For all $\eps$, let $x(\eps)$ and $y(\eps)$ be smooth vector valued functions determining right and left eigenvectors of $A + \eps E (\eps)$ of unit norm and such that $y(\eps) ^H x(\eps) > 0$, and $G(\eps)$ the G-inverse of $A + \eps E(\eps) - \lambda(\eps) I$. 
For the function $r(\eps)$ we have,
\begin{eqnarray}
\frac{\mathrm{d} r(\eps)}{\mathrm{d} \eps} & = &  r(\eps) \Re \Bigl(x(\eps) ^HG(\eps) E(\eps) x(\eps) + y(\eps) ^H E(\eps) G(\eps) y(\eps) \Bigr) \le 0\;.
\label{eq:der}
\end{eqnarray}
\end{theorem}

\begin{proof}
We indicate by $'$ differentiation with respect to $\eps$ and get 
\begin{eqnarray}
r'(\eps) & = & y'(\eps) ^H x(\eps) + y(\eps) ^H x'(\eps)\;.
\label{eq:dr}
\end{eqnarray}
Making use of \eqref{ms} and recalling that $G(\eps)x(\eps)\equiv 0$ and $y(\eps) ^HG(\eps)\equiv 0$ we obtain 
\begin{eqnarray}
r'(\eps)  & = & r(\eps) \Re\Bigl( x(\eps) ^H G(\eps) E(\eps) x(\eps) + y(\eps) ^H E(\eps) G(\eps) y(\eps) \Bigr) \nonumber
\\ && + \,  \eps\,r(\eps) \Re\Bigl(  x(\eps) ^H G(\eps) E'(\eps) x(\eps) + y(\eps) ^H E'(\eps) G(\eps) y(\eps) \Bigr)\;.
\label{eq:zc}
\end{eqnarray}
In order to prove the theorem we have to show that the second addendum in \eqref{eq:zc} vanishes. By Theorem \ref{thm:1} and the analogous for the real case we have that the extremizer (stationary point of the corresponding ODE) is 
\be
E(\eps) = c\,S(\eps), \qquad \mbox{where} \ 
S(\eps) = y(\eps) y(\eps) ^H G(\eps) ^H + G(\eps) ^H x(\eps) x(\eps) ^H,
\label{eq:Eeps}
\ee
for a suitable constant $c$. Furthermore, by norm conservation of $E(\eps)$ as $\eps$ varies, 
\be
\Re \big\langle E(\eps), E'(\eps) \big\rangle = 0\;.
\label{eq:EEp}
\ee

Equations (\ref{eq:Eeps}) and (\ref{eq:EEp}) imply (see (\ref{eq:zc}))
\begin{eqnarray*}
r'(\eps) = r(\eps) \Re \Bigl( \langle S(\eps), E(\eps) \rangle + \eps \langle S(\eps), E'(\eps) \rangle \Bigr)
= r(\eps) \Re \langle S(\eps), E(\eps) \rangle\;,
\end{eqnarray*}
that is (\ref{eq:der}). The non-positivity is due to monotonicity of $r(\eps)$ with respect to $\eps$.
This property follows immediately by noticing that the minimum in \eqref{eq:reps} can be equivalently computed on the closed ball $\{E\colon \| E \|_F \le 1\}$. Indeed, given a matrix $E$ of norm $\|E\|_F<1$, the matrix $E_\alpha = E + \alpha (I - xy ^H/y ^Hx)$ is such that $A+\eps E_\alpha$ has the same right and left eigenvectors $x$ and $y$, and the parameter $\alpha$ can be chosen to have $\|E_\alpha\|_F=1$.
\end{proof}

\begin{remark}\rm
\label{ES}
It is worthwhile to notice that the monotonicity property of $r(\eps)$ implies that the constant $c$ in \eqref{eq:Eeps} is negative, i.e., $c=-\|S(\eps)\|_F^{-1}$. Indeed, consider the Cauchy problem,
\[
\begin{cases} \dot {\mc E}(t) = - \mc S(t)\;, \\ \mc E(0) = E(\eps)\;, \end{cases}
\]
where $\mc S(t) = y(t)y(t)^HG(t)^H+G(t)^Hx(t)x(t)^H$ is the maximal descent direction associated to $A+\eps\mc E(t)$, so that $y(t)^Hx(t) < r(\eps)$ for $t$ positive and small. On the other hand, as
$E(\eps) = c\, S(\eps)$, 
\[
\frac{\mathrm{d}}{\mathrm{d}t} \|\mc E(t)\|_F^2 \bigg|_{t=0} = 2c\|S(\eps)\|_F^2\;.
\]
Hence, $c>0$ would imply both $\|\mc E(t)\|_F > \|E(\eps)\|_F$ and $y(t)^Hx(t) < r(\eps)$ for $t$ positive and small, in contradiction with the non-increasing property of $r(\eps)$.

In particular, \eqref{eq:der} can be written in the more concise form,
\be
\label{r's}
r'(\eps) = - r(\eps) \|S(\eps)\|_F\;.
\ee
\end{remark}

We may use Theorem \ref{th:der} to devise a Newton-bisection iteration to solve the equation $r(\eps)=\delta$.
However, since $r(\eps)$ is singular at $\eps=\eps^{0,*}$, this is not recommended for small $\delta$ (as our
experiments have put in evidence). 

We now analyze the behavior of $r(\eps)$ as $\eps$ approaches $\eps^{0,*}$. For the sake of simplicity we only consider the complex case, but similar results can be deduced in the real structured case. To state and prove the main result some extra remarks are needed. 

A first remark is that, for the generic case under consideration in which only two eigenvalues coalesce to form a Jordan block, the matrix $A+\eps^{0,*} E(\eps^{0,*})$ is non-derogatory and hence it is a continuity point for its invariant subspaces \cite{CH80}. In particular,
\be
\label{xy*}
x_*=x(\eps^{0.*}) = \lim_{\eps\nearrow\eps^{0,*}} x(\eps)\;, \qquad y_*=y(\eps^{0.*}) = \lim_{\eps\nearrow\eps^{0,*}} y(\eps)\;,
\ee
where, obviously, $y_*^Hx_* = r(\eps^{0,*})=0$. For the same reason, letting now,
\be
\label{M}
M(\eps) = A+\eps E(\eps)\;,
\ee
the rank of the matrix $M(\eps)-\lambda(\eps)I$ remains equal to $n-1$ also at $\eps=\eps^{0,*}$. Therefore, denoting by $Q^\dag$ the Moore-Penrose pseudoinverse of the matrix $Q$ and setting
\be
\label{BC}
B(\eps)=(M(\eps)-\lambda(\eps)I)^\dag\;, \qquad C(\eps) = y(\eps)^H B(\eps)x(\eps)\;,
\ee
the following limits exist and are finite,
\be
\label{BCl}
B_*=B(\eps^{0.*}) = \lim_{\eps\nearrow\eps^{0,*}} B(\eps)\;, \qquad C_*=C(\eps^{0.*}) = \lim_{\eps\nearrow\eps^{0,*}} C(\eps)\;.
\ee

\begin{proposition}
\label{prop:boh}
In the same hypothesis of Theorem \ref{th:der}, assume that the branch of minima extends to the whole interval $[\underline\eps,\eps^{0,*})$. Then,
\be
\label{rr'}
\lim_{\eps\nearrow\eps^{0,*}}r(\eps)r'(\eps) = -2 |C_*| \;.
\ee
\end{proposition}

\begin{proof}
We recall that the group inverse $G$ of a matrix $Q$ which has nullity one admits the following representation \cite{GO11},
\[
G = \Big(I-\frac{xy^H}{y^Hx}\Big) \, Q^\dag\, \Big(I-\frac{xy^H}{y^Hx}\Big)\;,
\]
where  $x,y$ are right and left null vectors of $Q$, respectively. In our context this gives,
\be
\label{Ge}
\begin{split}
G(\eps) & = \left(I-\frac{x(\eps)y(\eps)^H}{r(\eps)}\right) \, B(\eps) \, \left(I-\frac{x(\eps)y(\eps)^H}{r(\eps)}\right) \\ & = \frac{C(\eps)x(\eps)y(\eps)^H}{r(\eps)^2}  - \frac{B(\eps)x(\eps)y(\eps)^H+x(\eps)y(\eps)^HB(\eps)}{r(\eps)} +B(\eps)\;,
\end{split}
\ee
where $B(\eps)$ and $C(\eps)$ are defined in \eqref{BC}. By \eqref{Ge} and using the following properties of the Moore-Penrose pseudoinverse, 
\be
\label{Be}
x(\eps)^HB(\eps) = 0\;,\qquad B(\eps)y(\eps)= 0\;,
\ee
we have
\be
\label{SH}
S(\eps)^H = \frac{2C(\eps)x(\eps)y(\eps)^H}{r(\eps)^2} - \frac{B(\eps)x(\eps)y(\eps)^H+x(\eps)y(\eps)^HB(\eps)}{r(\eps)}\;.
\ee
By \eqref{SH} and \eqref{BCl} we get,
\be
\label{stS}
\lim_{\eps\nearrow\eps^{0,*}} r(\eps)^2\|S(\eps)\|_F = 2|C_*|\;.
\ee
Equation \eqref{rr'} now follows by \eqref{r's} and \eqref{stS}.
\end{proof}

The above proposition shows that the function $r(\eps)$ approaches zero like $\sqrt{\eps^{0,*}-\eps}$ as $\eps\nearrow\eps^{0,*}$. Indeed, under some further hypothesis (and however generically), such behavior holds true also for the higher order term in the Puiseaux expansion of $r(\eps)$ at $\eps^{0,*}$. This is the content of the next proposition, whose proof is in Appendix \ref{sec:A}.

\begin{proposition}
\label{prop:boh1}
In the same hypothesis of Theorem \ref{th:der}, assume that the branch of minima extends to the whole interval $[\underline\eps,\eps^{0,*})$. Then, if $\eps^{0,*}\|B_*\|_F$ is sufficiently small,
\be
\label{r3r''}
\lim_{\eps\nearrow\eps^{0,*}}r(\eps)^3r''(\eps) = -2|C_*|(1+4|C_*|)\;. 
\ee
\end{proposition}

\subsection{The algorithm} 
\label{sec:7.1}

For $\eps$ close to $\eps^{0,*}$, $\eps < \eps^{0,*}$, we have generically, 
\begin{eqnarray}
&&
\left\{
\begin{array}{rcl}
r(\eps) & = & \gamma \sqrt{\eps^{0,*} - \eps} + \bigo\bigl( (\eps^{0,*} - \eps)^{3/2} \bigr) \;,
\\[2mm]
r'(\eps) & = & \displaystyle{-\frac{\gamma}{2 \sqrt{\eps^{0,*} - \eps}}} + 
\bigo\bigl( (\eps^{0,*} - \eps)^{1/2} \bigr)\;,
\end{array}
\right.
\label{eq:eps}
\end{eqnarray}
which corresponds to the coalescence of two eigenvalues. 
In order to set up an iterative process, given $\eps_k$, we use Theorem \ref{th:der} to compute $r'(\eps)$ and estimate 
$\gamma$ and $\eps^{0,*}$ by solving (\ref{eq:eps}) with respect to $\gamma$ and 
$\eps^{0,*}$. We denote the solution as $\gamma_k$ and $\eps^{0,*}_k$, i.e.,
\begin{eqnarray}
\gamma_k & = & \sqrt{2 r(\eps_k) |r'(\eps_k)|}\;, \qquad
\eps^{0,*}_k = \eps_k + \frac{r(\eps_k)}{2 |r'(\eps_k)|}\;,
\label{eq:stepk}
\end{eqnarray}
and then compute $\eps_{k+1} = \eps^{0,*}_k - {\delta^2}/{\gamma_k^2}$. An algorithm based on previous formul\ae \ is Algorithm \ref{algo}.
\begin{algorithm}
\DontPrintSemicolon
\KwData{$\delta$, ${\rm tol}$, 
starting eigenvalue $\lambda_0$ (candidate to coalesce), 
$\eps_{\ell}$ (lower bound, default $0$), $\eps_r$ (upper bound, see Section \ref{sec:8.0}) and $\eps_0 \in (\eps_\ell.\eps_r)$}
\KwResult{$\eps^{\delta,*}$}
\Begin{
\nl Set $k=0$ and $r(\eps_0)=2\delta$ (just to enter the while loop)\;		
\nl \While{$|r(\eps_k) - \delta| \ge {\rm tol}$}{

\nl Compute $r(\eps_k)$ by integrating (\ref{eq:r2ode}) (complex case) or 
    (\ref{realeq:r2ode}) (real case), with initial datum $E(\eps_{k-1})$ (if $k \ge 1$)\;		
\nl \eIf{$r(\eps_k) > {\rm tol}$}{
		\nl Set $\eps_\ell=\eps_k$\;
    \nl Set ${\rm Bisect} = {\rm False}$} {
		Set $\eps_r=\eps_k$\;
    \nl Set ${\rm Bisect} = {\rm True}$ }

\nl \eIf{${\rm Bisect} = {\rm True}$}{
    Set $\eps_{k+1} = \left( \eps_\ell+\eps_r \right)/2$\;
    }{
Compute $r'(\eps_k)$ by (\ref{eq:der})\;     
\nl	Compute $\gamma_k$ and $\eps^{0,*}_k$ by (\ref{eq:stepk})\;
\nl Set $\widehat\eps_{k+1} = \eps^{0,*}_k - \displaystyle{\frac{\delta^2}{\gamma_k^2}}$\;
\nl \eIf{$\widehat\eps_{k+1} \in \left( \eps_\ell, \eps_r \right)$}{Set $\eps_{k+1} = \widehat\eps_{k+1}$}
{Set $\eps_{k+1} = \left( \eps_\ell+\eps_r \right)/2$}
    }
\nl Set $k=k+1$\;
}
\nl Print $\eps^{\delta,*} \approx \eps_k$\;
\nl Halt
}
\caption{Basic algorithm for computing the $\delta$-distance to defectivity. \label{algo}}
\end{algorithm}

The correction of $\eps_{k+1}$ in the if-statement at line {\bf 8} is due to the fact that 
if $\eps_{k+1}$ is larger than $\eps^{0,*}$ the value is corrected by a bisection step. 

The test at line {\bf 2} - when negative - means that for $\eps=\eps_k$ there are
coalescing eigenvalues (up to a tolerance {\rm tol}) in the $\eps$-pseudospectrum. 

The algorithm shows quadratic convergence, that is (for small $\delta$) and close to $\eps^{\delta,*}$,
\[
| \eps_{k+1} - \eps^{\delta,*} | = \bigo\bigl( | \eps_{k+1} - \eps^{\delta,*} |^2 \bigr)\;,
\] 
as we will show in the forthcoming section of numerical illustrations.
Clearly if $\delta$ is not too small a classical Newton iteration might also be alternatively used. 

\begin{remark}\rm
If we set $\delta$ sufficiently small, by exploiting the Puiseux expansion (\ref{eq:eps}),
we find at the same time an accurate solution $\eps^{\delta,*}$ of
equation $r(\eps)=\delta$ and of $\eps^{0,*}$, that is an upper bound for the distance to defectivity.
\end{remark}

\subsection{Illustrative examples}

We consider first few illustrative examples of small dimension, of both complex and real matrices. 
For the second real example we compute both the complex and the real distance to defectivity.

Then we report the results obtained on some test problems also of large size.

\subsection*{Example 1}

Consider the complex matrix 
$$A=\left( \begin{array}{rrrrr}
   0        &  1 +  \iu &   2 +  \iu &  1 + 2\iu &        1 \\						
  -1        & -1 -  \iu &   1 -  \iu &    -  \iu &        0 \\          
   1 -  \iu & -1 - 2\iu &   1 + 2\iu &    - 2\iu &        0 \\          
   1 - 2\iu &  1 -  \iu &  -1 + 2\iu & -1 -  \iu &        0 \\          
   1        & -1 -  \iu &       2\iu & -1 -  \iu &   - 2\iu
\end{array} \right)\,.
$$
\begin{figure}[ht]
\centerline{
\includegraphics[width=7cm]{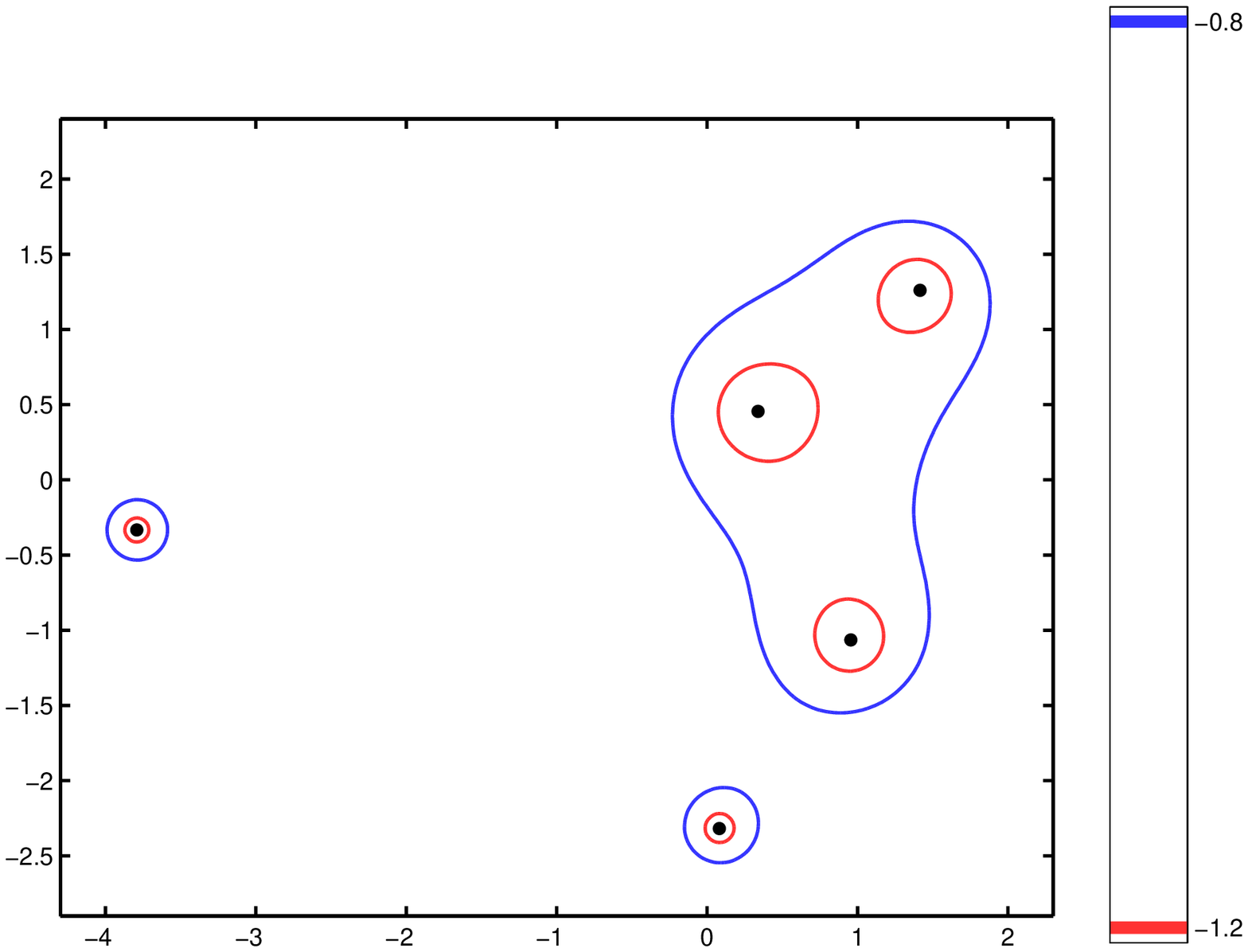} \hskip 0.25cm \includegraphics[width=4.3cm]{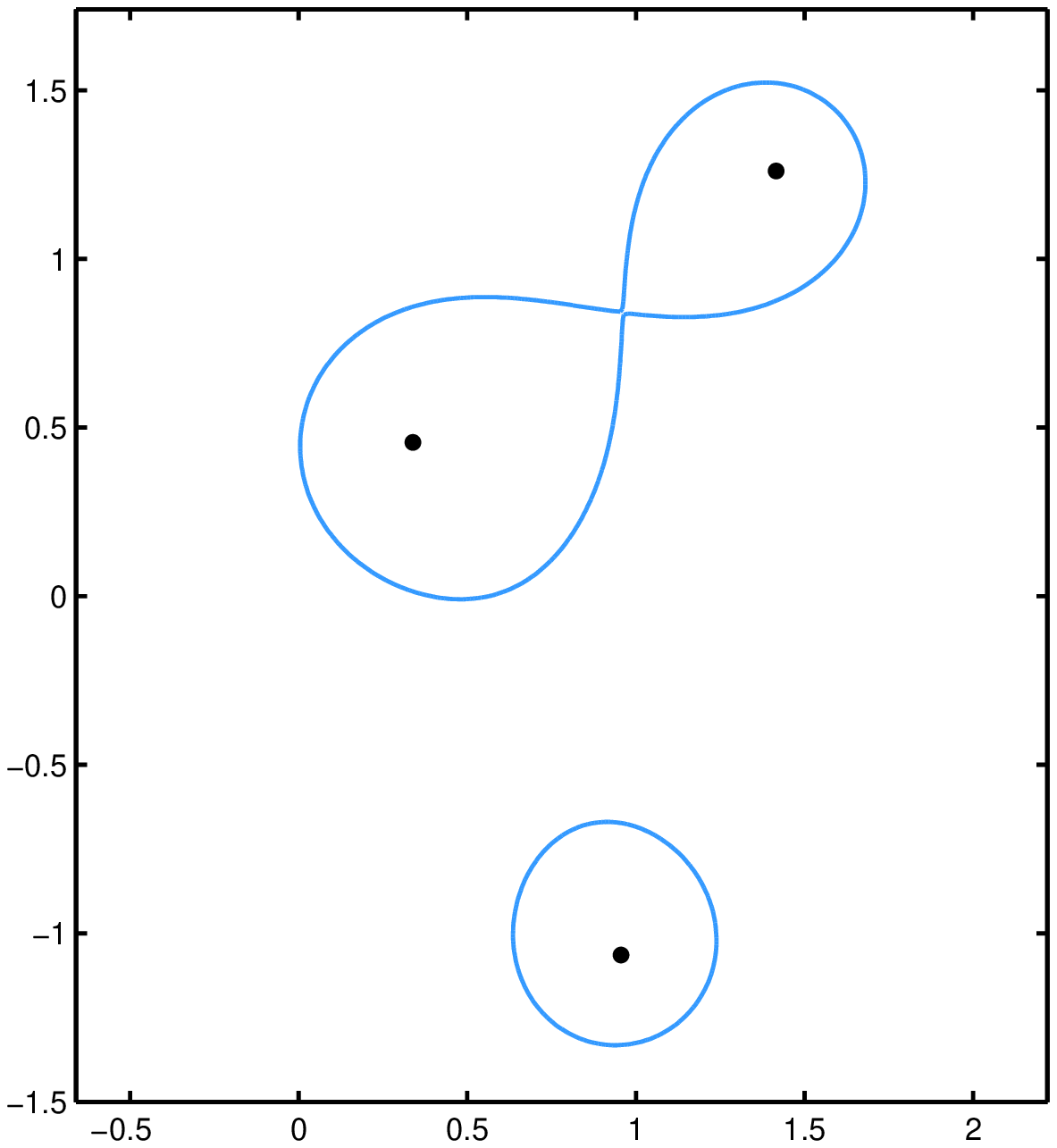}}
\caption{Left picture: the complex $\eps$-pseudospectrum for $\eps_1=10^{-1.2}$ and $\eps_2=10^{-0.8}$ 
indicates that the distance to defectivity $\eps^{0,*} \in (\eps_1,\eps_2)$. Right picture: the 
$\eps$-pseudospectrum for $\eps = \eps^{0,*}$; the coalescent eigenvalues appear in correspondence
of the intersections of the two ovals at $\lambda \approx 0.961516149290911 + 0.840702239813292 \iu$. \label{Ex_1}}
\end{figure}

By using the procedure described in \cite{ABBO11} we identify as initial eigenvalues,
candidate to coalesce, 
\begin{eqnarray*}
\lambda_1 = 1.416177710 + 1.260523165\,\iu\;, 
\quad
\lambda_2 = 0.338991381 + 0.455810180\,\iu\;.
\end{eqnarray*}
\begin{table}[ht]
\begin{center}
\begin{tabular}{l|l|l|}\hline
 $k$ & $\eps_k$ & $r(\eps_k)$  \\
 \hline
\rule{0pt}{9pt}
\!\!\!\! $0$ & $0.063095734448019$ & $0.105255211077608$   \\
 $1$         & $0.086013666854219$ & $< {\rm tol}$         \\
 $2$         & $0.081430080372979$ & $0.031280704783587$   \\
 $3$         & $0.082922950147539$ & $< {\rm tol}$         \\
 $4$         & $0.082624376192627$ & $0.013324941846019$   \\
 $5$         & $0.082879786721242$ & $< {\rm tol}$         \\
 $6$         & $0.082828704615519$ & $0.005919163134962$   \\
 $7$         & $0.082876946962636$ & $0.000910106101987$   \\
 $8$         & $0.082876706789675$ & $0.000999989689847$   \\
 $9$         & $0.082876706760826$ & $0.000999999999761$   \\
 \hline
\end{tabular}
\vspace{2mm}
\caption{Computed values of $\eps$ and $r(\eps)$ for Example 1. \label{tab:ex1}}
\end{center}
\end{table}

Applying the algorithm presented in this paper starting by $\lambda_1$, with $\eps_0=10^{-1.2}$,
$\delta=10^{-3}$, $\theta=0.8$  and ${\rm tol}=10^{-6}$, we compute
\[
\eps^{\delta,*} = 0.082876706760826\;.
\]

Table \ref{tab:ex1} illustrates the behaviour of Algorithm \ref{algo}.

\subsection*{Example 2}

Consider the real matrix $A$,
\begin{equation}
A=\left( \begin{array}{rrrrrr}
     1  &   1  &   1  &   1  &   0  &   0  \\
    -1  &   1  &   1  &   1  &   1  &   0  \\
     0  &  -1  &   1  &   1  &   1  &   1  \\
     0  &   0  &  -1  &   1  &   1  &   1  \\
     0  &   0  &   0  &  -1  &   1  &   1  \\
     0  &   0  &   0  &   0  &  -1  &   1
\end{array} \right)\,,
\label{eq:grcarmat}
\end{equation}
known as {\em Grcar}, of dimension $6$. A pair of initial eigenvalues candidate to coalesce is
\begin{eqnarray*}
\lambda_1 = 0.358489183 - 1.950114681\,\iu\;,
\quad
\lambda_2 = 1.139108055 - 1.230297560\,\iu\;.
\end{eqnarray*}
\begin{table}[ht]
\begin{center}
\begin{tabular}{l|l|l|}\hline
 $k$ & $\eps_k$ & $r(\eps_k)$  \\
 \hline
\rule{0pt}{9pt}
\!\!\!\! $0$ & $0.100000000000000$ & $0.589633247093566$   \\
 $1$         & $0.350780150979330$ & $< {\rm tol}$         \\
 $2$         & $0.300624120783464$ & $0.014186043007898$   \\
 $3$         & $0.300716631787740$ & $0.000976909835249$   \\
 $4$         & $0.300716233100507$ & $0.001348256831758$   \\
 $5$         & $0.300716610710709$ & $0.000999998097688$   \\
 $6$         & $0.300716610677479$ & $0.001000034080827$   \\
 $7$         & $0.300716610708953$ & $0.001000000000005$   \\
 \hline
\end{tabular}
\vspace{2mm}
\caption{Computed values of $\eps$ and $r(\eps)$ for Example 2 (real distance).\label{tab:ex2}
}
\end{center}
\end{table}

Applying the algorithms presented in this paper starting by $\lambda_1$, with $\eps_0=10^{-1}$,
$\delta=10^{-3}$, $\theta=0.8$  and ${\rm tol}=10^{-6}$, we compute the following values
\begin{eqnarray*}
\eps^{\delta,*} & = & 0.215185436319885 \qquad \mbox{(complex distance)}\;,
\\
\eps^{\delta,*} & = & 0.300716610708953 \qquad \mbox{(real distance)}\;. 
\end{eqnarray*}

Table \ref{tab:ex2} illustrates the behaviour of Algorithm \ref{algo} for computing the
real $\delta$-distance. 
The coalescence point is $\lambda = 0.756775621111013 - 1.594861012705232\,\iu$.

The estimated distance to defectivity is $\eps^{0,*} \approx 0.300725344809309$.

%
%
%
%

\section{Implementation issues}
\label{sec:8}

In this section we discuss some aspects relevant to the numerical computation of the
$\delta$-distance to defectivity.

We consider here the cases where the Jordan canonical form $A = V D V^{-1}$  is available ($D$ is the 
diagonal matrix of eigenvalues of $A$ and $V$ the matrix whose columns are the associated
eigenvectors). This case is interesting when the size of $A$ is not too large.
The situation where Jordan canonical form of $A$ is not available would be also interesting but it needs a 
guess for the candidate eigenvalues to coalesce and for this reason we do not discuss it here.

\subsection{Upper bound}
\label{sec:8.0}

In order to compute an upper-bound $\eps_r$ to $\eps^{\delta,*}$, we consider the canonical decomposition of $A$,
\[
A = \sum\limits_{i=1}^{n} \lambda_i \frac{1}{y_i^H x_i} x_i y_i^H\;,
\]
where $x_i$ and $y_i$ are the right and left eigenvectors of $A$ associated to $\lambda_i$, normalized as 
$\| x_i \| = \| y_i \| = 1$ and $y_i^H x_i > 0$ for all $i$.
In the case of complex perturbations ($\K=\C$) we have the following natural upper bound,
\begin{equation}
\eps_r^\C = \min\limits_{1 \le i \le n} \min\limits_{1 \le j \le n, j\neq i} \frac{|\lambda_i -\lambda_j|}{y_i^H x_i}\;.
\label{eq:boundC}
\end{equation}
For a real matrix, in the case of real admissible perturbations, we have the same bound (\ref{eq:boundC}) if we 
consider coalescence of two real eigenvalues either two complex conjugate eigenvalues. Otherwise, if two complex
eigenvalues (not conjugate to each other) coalesce, we have to consider a double coalescence (due to the fact that
also the coonjugate pair coalesces). This gives the following value,
\begin{equation}
\eps_r^\R = \min\limits_{1 \le i \le m} \min\limits_{1 \le j \le m, j \neq i} 2\frac{|\Re\left(\lambda_i -\lambda_j\right)|}{y_i^H x_i}\;,
\label{eq:boundR}
\end{equation}
where the set of the first $m\le n/2$ eigenvalues have nonzero imaginary part and do not contains any conjugate pair.

\subsection{Choice of the initial eigenvalue}
\label{sec:8.1}

We make use of the method proposed in \cite{ABBO11}, which we extend to the case of real
perturbations. 

Let $p_j$ be the eigenvalue condition number for $\lambda_j$ of $A$, we define the 
candidate coalescence point, say $z_0$, by the following combination of two eigenvalues
$$z_0=\frac{p_{j'}\lambda_{k'} + p_{k}\lambda_{j'}}{p_{j'}+ p_{k'}}\;,
$$
where the index pair $(j', k')$ minimizes $|\lambda_j-\lambda_k|/(p_j+p_k)$ over all distinct pairs of eigenvalues. The starting guess is derived from first order perturbation bounds for simple eigenvalues. If $A$ is perturbed by $\eps E$ with $\eps$ small, then for an eigenvalue 
$\lambda_j$ of $A$ there exists an eigenvalue $\tilde \lambda_j$ of $A+\eps E$ such that 
$|\tilde \lambda_j -\lambda_j| \leq p_j\eps + \bigo (\eps^2)$. Thus, for sufficiently small $\eps$, the component of $\Lameps^{\K}(A)$ containing $\l_j$ is approximately a disk of radius $p_j \eps$ centered at $\lambda_j$. Therefore, a point of coalescence of two components of $\Lameps^{\K}(A)$ containing eigenvalues, say, $\lambda_j$ and $\lambda_k$, is expected to be approximated by the point of coalescence of the disks $|z-\lambda_j|\leq p_j\eps$ and $|z-\lambda_k| \le p_k \eps$. 
This gives the guess $z_0$ which is the point of coalescence of the disks for the eigenvalues $\lambda_{j'}$ and $\lambda_{k'}$.
Note that when $\K=\C$, the rate $p_j$  of $\lambda_j$  is given by the traditional eigenvalue condition number $\kappa(\lambda_j)$, whereas when  $\K = \R$, the role of $p_j$ is played by the first-order measure in the Frobenius norm of the worst-case effect on $\lambda_j$ of real perturbations; see, e.g., \cite{BK04, KKT06}. 


\subsection{Eigenvalue computation}
\label{sec:8.2}

For problems of small dimension we need to compute the eigenvalues and the eigenvectors of the matrices 
$A + \eps E$ by using the Matlab routine \emph{eig}  and in particular, since the relevant eigenvalues and eigenvectors are often ill conditioned, once we have the right eigenvector, we do not compute the left one only by inverting the matrix, but we make a second call to \emph{eig} in order to compute the right eigenvectors of $(A + \eps E)^H$ (similarly to what is done in \cite{GO11}).


For problems of large dimension (and possibly sparse structure)  we use the routine \emph{eigs}, which is an interface for ARPACK \cite{LSY98}, a code for implementing the implicitly restarted Arnoldi method instead of \emph{eig}. This choice is based on the fact that  \emph{eigs} accepts as input a sparse matrix and function handles and therefore we do not need to compute the dense matrix $A+\eps E$ explicitly. 

\subsection{Computation of the group inverse}
\label{sec:8.3}

The main computational problem when computing the right-hand side of the differential equations 
(\ref{eq:r2ode}) and (\ref{realeq:r2ode}) is the the application of the group inverse $G$ to a vector. 
In order to compute it efficiently we make use of the following result from \cite{GO13}.
\begin{theorem} \label{th:Ginv}
Suppose that $B$ has index one with $x \in \ker(B)$
and $y \in \ker(B^H)$ of unit norm and such that $y^H x > 0$.  
Let $B^\#$ be the group inverse of $B$. Then
\begin{eqnarray}
B^\# & = & \Pi B\pin \Pi\;,
\label{eq:groupinvformula}
\\
B^\# & = & \Pi \left(  B + y x^H \right)^{-1} \Pi\;,
\label{eq:f2}
\end{eqnarray}
where $\Pi$ is the projection $\displaystyle{\left( I -  x z^H \right) }$
with $z = y/(y^H x)$.
\end{theorem}

In the case considered here we can apply Theorem \ref{th:Ginv} with
$B = A -\lambda I + \eps E$, so that - using (\ref{eq:f2}) - $G = \Pi \left( A -\lambda I + F \right)^{-1} \Pi$, where
$
F = \eps E + \mu x z^H = U_1 \Sigma_1 V_1^H 
$
is a rank-$\ell$ matrix,
with $\ell=3$ in the complex case and $\ell=5$ in the real case.
Thus $U_1$ and $V_1$ are $n \times \ell$ matrices and $\Sigma_1$ is a $\ell \times \ell$ 
diagonal matrix.
Since $\left( A - \lambda I \right)^{-1} = V \left(D - \lambda I\right)^{-1} V^{-1}$ and we
have $V$, $D$ and $V^{-1}$ (which we have computed in the beginning to detect a candidate eigenvalue of $A$ to coalesce), we can make use of the well-known Sherman-Morrison-Woodbury formula \cite{W50} 
to exploit the low-rank structure of $F$, i.e., $\left(A -\lambda I + F\right)^{-1} $ is given by
\[
\left(A -\lambda I \right)^{-1} - \left(A - \lambda I \right)^{-1} U_1 \Sigma_1 
\left(\Sigma_1 + \Sigma_1 V_1 \left(A - \lambda I \right)^{-1} U_1 \Sigma_1 \right)^{-1} 
\Sigma_1 V_1 \left(A - \lambda I \right)^{-1}.
\]
%
%
This allows for an efficient computation of the matrix vector product $G w$, $w$ being an arbitrary vector. 
When the matrix $A$ is sparse we can make use of a sparse linear solver. For example
we can use GMRES as the computation of matrix products of the form $A+F$ with $A$ sparse and $F$ of
low rank is inexpensive.
Additionally, making use of  the routine \emph{ilu}, we have the sparse incomplete LU factorization of $A$, i.e., $\widetilde{L}\widetilde{U}$, and  we can compute as before the vector product $\widetilde{G} w$, $w$ being an arbitrary vector and $\widetilde{G}$ the group inverse of $\widetilde{L}\widetilde{U} + F$,  as a preconditioner to be passed to  the \emph{gmres}. 

\subsection{Scaling of the right-hand side in the ODEs}
\label{sec:8.4}

When $\eps=\eps^{0,*}$ the group inverse of $A+\eps E(\eps) -\lambda(\eps)I$ is not defined, and in fact in the considered ODEs it becomes singular as time approaches infinity, i.e., close to coalescence. This means that when $\eps^{0,*}-\eps$ is very small the group inverse gets very large at the stationary points. To overcome this problem we may premultiply the right-hand side of the ODEs by the real scalar $r^2=(y^Hx)^2$. Indeed, this is equivalent to replace $G$ by $(y^Hx)^2G$, which has a finite limit as $\eps\nearrow\eps^{0,*}$, that is, by \eqref{BC}, \eqref{BCl}, and \eqref{Ge}, 
\[
\lim_{\eps\nearrow\eps^{0,*}} r(\eps)^2 G(\eps) = \lim_{\eps\nearrow\eps^{0,*}} C(\eps)x(\eps)y(\eps)^H = C_* x_*y_*^H\;,
\]
where $x_*,y_*$ are defined in \eqref{xy*}.
 
Indeed, recalling (\ref{eq:f2}) the constant $C_*$ defined in \eqref{BCl}, is equivalently given by
\[
C_* = y_*^H \left( B(\eps^{0,*})  + y_* x_*^H \right)^{-1} x_*\;,
\] 
where $B(\eps) = A+\eps E(\eps) -\lambda(\eps)I$.
We get this showing the existence of the limit of 
$\left( B(\eps)  + y(\eps) x(\eps)^H \right)^{-1}$ as $\eps\nearrow\eps^{0,*}$.
To this purpose, consider the $2 \times 2$ block
\[
\left( \begin{array}{cc}
a & 1 \\
-a^2 + \delta a & - a + \delta
\end{array}
\right)\;,
\] 
where $a$ and $\delta$ are complex numbers depending on $\eps$, such that $\lim\limits_{\eps\nearrow\eps^{0,*}} \delta = 0$ and $\lim\limits_{\eps\nearrow\eps^{0,*}} a = a_* \in \R$. This is a generic matrix with a simple zero eigenvalue for $\eps < \eps^{0,*}$ which becomes double and defective at $\eps = \eps^{0,*}$.
An explicit computation yields
\[
B(\eps^{0,*})  + y_* x_*^H = \left( \begin{array}{cc}
2 a_* & 1-a_*^2 \\
1-a_*^2  & - 2 a_*
\end{array}
\right)\,,
\]
which is clearly invertible.


\subsection{Euler discretization}
\label{sec:8.5}

The stepsize control in the numerical integration of the ODEs by Euler method is simply driven by the
monotonicity requirement $r(t_{n+1}) < r(t_n)$.
The stepsize $h$ can
be selected by the Algorithm \ref{algo2}.

\begin{algorithm}
\DontPrintSemicolon
\KwData{the eigenvalue $\lambda_n$ and the normalized eigenvectors $x_{n}, y_{n}$ of the matrix $A+\eps E_{n}$, 
(with $E_{n} \approx E(t_{n})$, $t_n=t_{n-1}+h_{n-1}$), \ $\sigma>1$ (stepsize ratio), $h$ (predicted stepsize)}
\KwResult{$h_{n+1}, \,E_{n+1}, \, \lambda_{n+1}$}
\Begin{
\nl Apply a step of Euler's method to (\ref{eq:r2ode}) (or (\ref{realeq:r2ode})) to obtain 
    $\widetilde E_{n+1} \approx E(t_n + h)$.\;
\nl Compute the eigenvalue $\widetilde\lambda(h)$ of $A+\eps \widetilde E_{n+1}$, closer to $\lambda_{n}$.\;
\nl Compute $\widetilde{x}, \widetilde{y}$ be the normalized eigenvectors associated to $\widetilde\lambda(h)$. \;					
\nl \If{$\widetilde{y}^H \widetilde{x} \ge y_n^H x_n$}{
    Reduce the stepsize $h=h/\sigma$,\; 
    Repeat from 1.}
\nl \If{$h \ge h_{n-1}$}{
    Apply a step of Euler's method to obtain $\widehat E_{n+1} \approx E(t_n + \sigma h)$.\;
    Compute $\widehat\lambda(t_n + \sigma h)$ and the associated normalized eigenvectors
    $\widehat{x}, \widehat{y}$\;
		   \If{$\widehat{y}^H \widehat{x}  \le \widetilde{y}^H \widetilde{x}$}{
		   Increase the stepsize $h=\sigma h$}
		}
\nl	Set $\lambda_{n+1}= \widetilde\lambda(t_n+h)$, $E_{n+1} = \widetilde E_{n+1}$\;
}
\caption{Stepsize selection. \label{algo2}}
\end{algorithm}

\section{Numerical illustrations: complex and real-structured distances}
\label{sec:9}

\subsection*{Example 3: {\tt West0067} matrix, dimension \texorpdfstring{$67$}{67}} 

This real matrix is part of the Chemwest collection, modeling of chemical engineering plants.

\paragraph*{Complex distance}

We set $\delta=10^{-3}$ and obtain the 
following values:
\[
\eps^{0,*} = 0.00551675100\;, \qquad \gamma = 2.01525876223\;, \qquad \eps^{\delta,*} = 0.00551650477\;.
\]
The computed eigenvalues, which are expected to coalesce for $\eps=\eps^{0,*}$ are
\[
\lambda_1 = -0.252310074 + 0.853692119\iu\;, \quad
\lambda_2 = -0.252142643 + 0.853721574\iu\;,
\]
having a distance $|\lambda_1 - \lambda_2| \approx 1.7\cdot 10^{-4}$.

These values agree with those computed by the code in \cite{ABBO11}, i.e.,
$\eps^{0,*} = 0.00551675$.

\paragraph*{Real distance}

We set $\delta=10^{-3}$ and obtain the results in Table \ref{tab:exwest67r}; as a by-product we compute the 
following values:
\[
\eps^{0,*} = 0.0078798150\;, \qquad \gamma = 0.2864276818\;, \qquad \eps^{\delta,*} = 0.0078676260\;.
\nonumber
\]
The $4$ computed eigenvalues, which are expected to coalesce pairwise for $\eps=\eps^{0,*}$ are
\begin{eqnarray}
&& 
\lambda_1 = -0.251808570 + 0.853267804\iu\;, \quad
\lambda_2 = -0.251792424 + 0.853170539\iu\;,
\nonumber
\end{eqnarray}
and their conjugates, having a distance $|\lambda_1 - \lambda_2| \approx 9.8\cdot 10^{-5}$.

\begin{table}[ht]
\begin{center}
\begin{tabular}{l|l|l|}\hline
 $k$ & $\eps_k$ & $r(\eps_k)$  \\
 \hline
\rule{0pt}{9pt}
\!\!\!\! $0$ & $0.010000000000000$ & $< {\rm tol}$   \\
 $1$         & $0.007500250000000$ & $0.032158858474852$   \\
 $2$         & $0.007871693262746$ & $< {\rm tol}$   \\
 $3$         & $0.007778832447060$ & $0.015901620958740$   \\
 $4$         & $0.007867876984625$ & $0.000534814297427$   \\
 $5$         & $0.007859174170350$ & $0.005003577527752$   \\
 $6$         & $0.007867628372516$ & $0.000996613906802$   \\
 $7$         & $0.007867625995876$ & $0.000999997977039$   \\
 \hline
\end{tabular}
\vspace{2mm}
\caption{Computed values of $\eps$ and $r(\eps)$ for Example 3 (real-structured distance).\label{tab:exwest67r}}
\end{center}
\end{table}

\subsection*{Example 4: {\tt Orr-Sommerfeld matrix}, dimension \texorpdfstring{$99$}{99}}

This matrix is taken from EigTool and arises in the discretization of the Orr-Sommerfeld operator at large Reynolds numbers.
Since the matrix is complex we consider its real part. 

\paragraph*{Real distance}

We set $\delta=10^{-4}$ and obtain the results in Table \ref{tab:exOrrSomr}. Since two real eigenvalues coalesce, the real and complex distances give the same estimate. In particular we get
\begin{eqnarray}
&& \eps^{0,*} = 7.56829388628 \cdot10^{-4}\;, \quad \gamma = 55.720252607\;, \quad \eps^{\delta,*} = 0.00075682939\;.
\nonumber
\end{eqnarray}
The computed eigenvalues, which are expected to coalesce pairwise for $\eps=\eps^{0,*}$ are
\begin{eqnarray}
&& 
\lambda_1 = -0.2675984616 \cdot 10^{-5}\;, \quad
\lambda_2 = -0.2676113453 \cdot 10^{-5}\;.
\nonumber
\end{eqnarray}
whose distance is $1.3\cdot 10^{-10}$.

\begin{table}[ht]
\begin{center}
\begin{tabular}{l|l|l|}\hline
 $k$ & $\eps_k$ & $r(\eps_k)$  \\
 \hline
\rule{0pt}{9pt}
\!\!\!\! $0$  & $0.000500000000000$ & $0.622086808315294$   \\
 $1$          & $0.000875000000000$ & $< {\rm tol}$   \\
 $2$          & $0.000781250000000$ & $< {\rm tol}$   \\
 $3$          & $0.000710937500000$ & $0.015901620958740$   \\
 $4$          & $0.000765252777937$ & $< {\rm tol}$   \\
 $5$          & $0.000751673958453$ & $0.123760620457050$   \\
 $6$          & $0.000756985630150$ & $0.000022826360137$   \\
 $7$          & $0.000755657712225$ & $0.059797479397155$   \\
 $8$          & $0.000756841673737$ & $0.000999999116306$   \\
 $9$          & $0.000756759362406$ & $0.014719249656511$   \\
 $10$         & $0.000756829515328$ & $0.000626846172482$   \\
 $11$         & $0.000756811977097$ & $0.007346284471003$   \\
 $12$         & $0.000756829400745$ & $0.000193924478176$   \\
 $13$         & $0.000756829387268$ & $0.000104977652463$   \\
 $14$         & $0.000756829385407$ & $0.000100003485058$   \\
 \hline
\end{tabular}
\vspace{2mm}
\caption{Computed values of $\eps$ and $r(\eps)$ for Example 4 (real-structured distance).\label{tab:exOrrSomr}}
\end{center}
\end{table}

Interestingly, the value computed by the code in \cite{ABBO11} is
$\eps^{0,*} = 0.00551675$ and the two closest computed eigenvalues of the computed closest
defective matrix differ by $0.00032389$ and have condition numbers $3.96107$ and $4.18569$.

\subsection*{Example 5: {\tt str\_\_600} matrix, dimension \texorpdfstring{$363$}{363}}

This real matrix is part of the SMTAPE collection, and arises from an application of the simplex method.

The complex distance to nearest defective matrix found by the code in \cite{ABBO11} is $6.80846\cdot 10^{-5}$.

\paragraph*{Real distance}

We set $\delta=10^{-3}$ and obtain the results in Table \ref{tab:exstr600r}; as a by-product we compute the 
following values:
\[
\eps^{0,*} = 1.133466479\cdot 10^{-4}\;, \qquad \gamma = 0.237681872\;, \qquad \eps^{\delta,*} = 9.5645234943\cdot 10^{-5}\;.
\]
The computed eigenvalues, which are expected to coalesce pairwise for $\eps=\eps^{0,*}$ are complex conjugate in this case,
\[
\lambda_1 =  0.008426530 + 0.0021636406\iu\;, \qquad
\lambda_2 =  0.008426530 - 0.0021636406\iu\;,
\]
having a distance $|\lambda_1 - \lambda_2| \approx 4.3\cdot 10^{-3}$.

\begin{table}[ht]
\begin{center}
\begin{tabular}{l|l|l|}\hline
 $k$ & $\eps_k$ & $r(\eps_k)$  \\
 \hline
\rule{0pt}{9pt}
\!\!\!\! $0$ & $0.000100000000000$ & $0.000800966492236$   \\
 $1$         & $0.000095636329862$ & $0.001000272860068$   \\
 $2$         & $0.000095645234943$ & $0.001000000062242$   \\
 \hline 
\end{tabular}
\vspace{2mm}
\caption{Computed values of $\eps$ and $r(\eps)$ for Example 5 (real-structured distance).\label{tab:exstr600r}}
\end{center}
\end{table}


\section{Extension to different structures}
\label{sec:10}

In this section we discuss the computation of the pattern-structured distance to defectivity.

We consider here a sparse matrix $A$ with sparsity pattern ${\rm P}$ 
(identifying the nonzero elements of $A$, that is $a_{ij}=0$ if ${\rm P}_{ij}=0$)  
and consider the manifold 
\begin{equation}
\label{Sp}
\Sp = \{ E\in \C^{n\times n} : \|E\|_F=1, E \ \mbox{has sparsity pattern} \ {\rm P} \}\;.
\end{equation} 
In order to obtain a differential equation for $E(t) \in \Sp$ along which a pair of eigenvalues of $A+\eps E(t)$ 
are getting closer so that the quantity $y(t)^H x(t)$ decreases with respect to $t$, we proceed similarly to the previous case; 
at $E \in \Sp$, let $y, x$ be left and right eigenvectors of $A+\eps E$, with $y^H x>0$,  relative to the  
eigenvalue $\lambda(t)$ of $A+\eps E(t)$ with highest condition number, and consider the differential equation 
\be
\label{odesparse}
\dot E = - P_\Sp(S) + \Re\lan E,P_\Sp(S) \ran E\;, \qquad E\in \Sp\;,
\ee
where $S$ is defined in \eqref{S} and $P_\Sp(S)$ denotes the orthogonal projection of $S$ onto $\Sp$, which is simply achieved by setting to zero all entries which do not belong to the nonzero pattern, and normalizing with respect to the Frobenius norm. It can be shown that
(\ref{odefull}) is characterized by the property that $r(t)=y(t)^H x(t)$ is monotonically decreasing along its solutions. 

Similarly to Theorem \ref{th:der}, we obtain a closed formula for the derivative of the function $r(\eps)$ with respect to $\eps$,
and an algorithm similar to Algorithm \ref{algo} is obtained. 

In order to choose the initial eigenvalue for the numerical computation of the structured distance to defectivity, that is to say the point of coalescence of two components of the structured $\eps$-pseudospectrum of $A$, as in Section \ref{sec:8.1} one  computes  the  zero-structured condition numbers of the eigenvalues of $A$; see, e.g., \cite{KKT06,NP06}.

Let us consider an illustrative example.

\subsection*{Example 6}

Consider again the Grcar matrix of dimension $6$ (\ref{eq:grcarmat}).

Table \ref{tab:Grcar} summarizes the computed upper bounds to the $\delta$-distance 
to defectivity (with $\delta=10^{-3}$), obtained applying Algorithm \ref{algo} to the matrix $A$, 
preserving different structures $\Stru$ (${\rm P} \C^{n,n}$ denotes
the set of complex matrices with sparsity pattern ${\rm P}$ and similarly ${\rm P} \R^{n,n}$ denotes the set of real 
matrices with sparsity pattern ${\rm P}$).
Other linear structures may be considered, like the Hessenberg or the Toeplitz but we limit our experiments to
complex/real perturbations preserving or not the pattern structure.
\begin{table}[ht]
\begin{center}
\begin{tabular}{|l|l|}\hline
Structure $\Stru$ & Computed bound for $\delta$-distance to defectivity  \\
 \hline
\rule{0pt}{9pt}
 $\!\!\C^{6,6}$             & $0.2151857\ldots$ \\
 $\R^{6,6}$                 & $0.3007253\ldots$ \\
 ${\rm P} \C^{6,6}$         & $0.6845324\ldots$ \\
 ${\rm P} \R^{6,6}$         & $0.9423366\ldots$ \\
 \hline
\end{tabular}
\vspace{2mm}
\caption{Computed upper bounds for the structured distance to defectivity for Example 6.\label{tab:Grcar}}
\end{center}
\end{table}

\noindent We also remark that, taking into account the Toeplitz structure of $A$ in the choice of the initial eigenvalue, the point of coalescence of two components of the structured $\eps$-pseudospectrum of $A$ should be approximated making use of the Toeplitz-structured  eigenvalue condition numbers; see, e.g., \cite{NP07, BGN12}. 

\appendix
\section{Proof of Proposition 7.4}
\label{sec:A}
\
Since $r'(\eps) = r(\eps) \Re \langle S(\eps), E(\eps) \rangle$ and, by \eqref{eq:Eeps}, \eqref{eq:EEp}, 
\[
\Re \big\langle S(\eps), E'(\eps) \big\rangle = c^{-1} \Re \big\langle E(\eps), E'(\eps) \big\rangle =0\;,
\]
we have,
\[
\begin{split}
r''(\eps) & = r'(\eps) \Re \langle S(\eps), E(\eps) \rangle + r(\eps) \Re \langle S'(\eps), E(\eps) \rangle + r(\eps) \Re \langle S(\eps), E'(\eps) \rangle \\ & = \frac{r'(\eps)^2}{r(\eps)} + r(\eps) \Re \langle S'(\eps), E(\eps) \rangle \;,
\end{split}
\]
where, by \eqref{SH},
\be
\label{Es}
E(\eps) = - \frac{\overline{C(\eps)}}{|C(\eps)|} y(\eps)x(\eps)^H + \bigO(r(\eps)) \;.
\ee
By \eqref{rr'}, \eqref{Es}, and the above computation, we get
\be
\label{se}
\lim_{\eps\nearrow\eps^{0,*}}r(\eps)^3r''(\eps) =  -2|C_*| - \Re \left(\frac{\bar C_*}{|C_*|} (Nx_*)^Hy_*\right)\;,
\ee
provided the limit
\be
\label{N}
N = \lim_{\eps\nearrow\eps^{0,*}} r(\eps)^4 S'(\eps)
\ee
exists. To show this, in what follows we analyze the limiting behavior of $S'(\eps)$ as $\eps$ approaches $\eps^{0,*}$.  

Recalling $C(\eps)$ is defined in \eqref{BC}, by \eqref{ms}, \eqref{Ge}, and \eqref{Be}, a straightforward computation gives, 
\[
\begin{split}
C'(\eps)  & = \frac{2y(\eps)^HM'(\eps)x(\eps)y(\eps)^HB(\eps)^2x(\eps)}{r(\eps)} - y(\eps)^HM'(\eps)B(\eps)^2x(\eps)  \\ & \quad + y(\eps)^H B(\eps)^2M'(\eps) x(\eps)+ y(\eps)^HB'(\eps)x(\eps)\;, \\ (xy^H)'(\eps) & = \frac{y(\eps)^H M' (\eps)x(\eps)}{r(\eps)} \big(B(\eps)x(\eps)y(\eps)^H +x(\eps)y(\eps)^HB(\eps)\big) \\ & \quad - B(\eps)M'(\eps)x(\eps)y(\eps)^H -x(\eps)y(\eps)^HM'(\eps)B(\eps)\;,
\end{split}
\]
Differentiating the right-hand side of \eqref{SH} and using the above expressions we obtain a linear equation for $S'(\eps)$. In view of \eqref{stS}, such equation reads, 
\be
\label{SH'}
\begin{split}
S'(\eps)^H &  = \left( \frac{y(\eps)^HM'(\eps)x(\eps)y(\eps)^HB(\eps)^2x(\eps)}{r(\eps)^3} - \frac{4C(\eps)r'(\eps)}{r(\eps)^3}\right) x(\eps)y(\eps)^H \\ & \qquad +  \frac{2C(\eps)y(\eps)^HM'(\eps)x(\eps)}{r(\eps)^3} \big(B(\eps)x(\eps)y(\eps)^H +x(\eps)y(\eps)^HB(\eps)\big) \\ & \qquad - \frac{2C(\eps)}{r(\eps)^2} \big(B(\eps)M'(\eps)x(\eps)y(\eps)^H +x(\eps)y(\eps)^HM'(\eps)B(\eps)\big) \\ & \qquad +\mc R(M'(\eps),\eps)\;, 
\end{split}
\ee
where the remainder $\mc R(M',\eps)$ is an affine function of $M'$ which is subdominant (as $\eps\nearrow\eps^{0,*}$) with respect to the other terms in the right-hand side; more precisely, there is a positive constant $C'$ such that
\[
\|\mc R(M',\eps)\|_F \le C' \left(\frac{\|M'\|_F}{r(\eps)} + \frac{1}{r(\eps)^3} \right) \qquad \forall\, \eps\in [\underline{\eps},\eps^{0,*})\;.
\]
By Remark \ref{ES}, $E(\eps) = -\|S(\eps)\|_F^{-1}S(\eps)$, so that
\[
M'(\eps) = E(\eps) + \eps E'(\eps) = E(\eps) + \frac{\eps}{\|S(\eps)\|_F} \big(\Re \langle E(\eps), S'(\eps) \rangle E(\eps) -S'(\eps)\big)
\]
Plugging \eqref{Es} and the previous relation in \eqref{SH'} and using \eqref{Be}, we obtain,
\be
\label{SH'2}
\begin{split}
S'(\eps)^H &  = \left( \frac{\Gamma(\eps) y(\eps)^HB(\eps)^2x(\eps)}{r(\eps)^3} - \frac{4C(\eps)r'(\eps)}{r(\eps)^3}\right) x(\eps)y(\eps)^H \\ & \qquad +  \frac{2C(\eps)\Gamma(\eps)}{r(\eps)^3} \big(B(\eps)x(\eps)y(\eps)^H +x(\eps)y(\eps)^HB(\eps)\big) \\ & \qquad + \frac{2\eps C(\eps)}{r(\eps)^2\|S(\eps)\|_F} \big(B(\eps)S'(\eps)x(\eps)y(\eps)^H +x(\eps)y(\eps)^HS'(\eps)B(\eps)\big) \\ & \qquad +\mc R_1(S'(\eps),\eps)\;, 
\end{split}
\ee
where 
\be
\label{Ga}
\begin{split}
\Gamma(\eps) & = y(\eps)^H \eps E'(\eps) x(\eps) = \frac{\eps}{\|S(\eps)\|_F}   y(\eps)^H\big(\Re \langle E(\eps), S'(\eps) \rangle E(\eps) -S'(\eps)\big)x(\eps) \\ & = -\eps\frac{1+\bigO(r(\eps))}{C(\eps)\|S(\eps)\|_F} \Im \big[C(\eps) y(\eps)^HS'(\eps) x(\eps)\big]
\\ & = \eps\frac{1+\bigO(r(\eps))}{C(\eps)\|S(\eps)\|_F} \Im \big[\bar C(\eps) x(\eps)^HS'(\eps)^H y(\eps)\big]
\end{split}
\ee
and the remainder $\mc R_1(S',\eps)$ is an affine function of $S'$ such that, for some $C''>0$,
\be
\label{R1}
\|\mc R_1(S',\eps)\|_F \le C'' \left(r(\eps) \|S'\|_F + \frac{1}{r(\eps)^3} \right) \qquad \forall\, \eps\in [\underline{\eps},\eps^{0,*})\;.
\ee
Multiplying both side of \eqref{SH'2} by $x(\eps)^H$ to the left, by $y(\eps)$ to the right, and using \eqref{Be} we get,
\be
\label{sc}
x(\eps)^HS'(\eps)^Hy(\eps) =  \frac{\Gamma(\eps) y(\eps)^HB(\eps)^2x(\eps)}{r(\eps)^3} - \frac{4C(\eps)r'(\eps)}{r(\eps)^3} + x(\eps)^H\mc R_1(S'(\eps),\eps) y(\eps)\;.
\ee
Therefore, \eqref{SH'2} reads,
\be
\label{SH'3}
\begin{split}
S'(\eps)^H &  = x(\eps)^HS'(\eps)^Hy(\eps) x(\eps)y(\eps)^H \\ & \qquad + \frac{2C(\eps)x(\eps)^HS'(\eps)^Hy(\eps)}{y(\eps)^HB(\eps)^2x(\eps)} \big(B(\eps)x(\eps)y(\eps)^H +x(\eps)y(\eps)^HB(\eps)\big)  \\ & \qquad + \frac{2\eps C(\eps)}{r(\eps)^2\|S(\eps)\|_F} \big(B(\eps)S'(\eps)x(\eps)y(\eps)^H +x(\eps)y(\eps)^HS'(\eps)B(\eps)\big) \\ & \qquad + \frac{8C(\eps)^2r'(\eps)}{y(\eps)^HB(\eps)^2x(\eps)r(\eps)^3} \big(B(\eps)x(\eps)y(\eps)^H +x(\eps)y(\eps)^HB(\eps)\big)\\ & \qquad +\mc R_2(S'(\eps),\eps)\;, 
\end{split}
\ee
where the remainder $\mc R_2(S',\eps)$ satisfies, for any $\eps\in [\underline{\eps},\eps^{0,*})$,
\be
\label{R2}
x(\eps)^H\mc R_2(S',\eps)y(\eps) = 0\;,\qquad \|\mc R_2(S',\eps)\|_F \le C''' \left(r(\eps) \|S'\|_F + \frac{1}{r(\eps)^3} \right)
\ee
(for some $C'''>0$). Equation \eqref{SH'3} must be completed with the condition on $x(\eps)^HS'(\eps)^Hy(\eps)$ coming from \eqref{Ga} and \eqref{sc}, i.e.,
\begin{eqnarray}
\label{eqsc}
x(\eps)^HS'(\eps)^Hy(\eps) & = & \eps\frac{y(\eps)^HB(\eps)^2x(\eps)(1+\bigO(r(\eps)))}{C(\eps)r(\eps)^2\|S(\eps)\|_F} \Im \big[\bar C(\eps) x(\eps)^HS'(\eps)^H y(\eps)\big]\nonumber \\ &&   - \frac{4C(\eps)r'(\eps)}{r(\eps)^3} + x(\eps)^H\mc R_1(S'(\eps),\eps) y(\eps)\;.
\end{eqnarray}
By \eqref{rr'}, the known terms in \eqref{SH'3} and \eqref{eqsc} diverge at most like $r(\eps)^{-4}$. Therefore, the limit \eqref{N} exists and can be computed provided the equation \eqref{SH'3} is solvable at the leading order as $\eps\nearrow \eps^{0,*}$. To this purpose, we write,
\[
S'(\eps) = \frac{N}{r(\eps)^4} \big(1+ \bigO(r(\eps))\big)\;.
\]
Then, recalling \eqref{xy*}, \eqref{BCl}, \eqref{rr'}, \eqref{stS}, and setting
\[
\alpha = \frac{2C_*}{y_*^HB_*^2x_*}\;, \qquad \beta = \lim_{\eps\nearrow \eps^{0,*}}\frac{2\eps C(\eps)}{r(\eps)^2\|S(\eps)\|_F}=\frac{\eps^{0,*}C_*}{|C_*|}\;,
\]
the equation \eqref{SH'3} implies that the unknown $N$ has to solve the linear equation,
\be
\label{NK}
\begin{split}
N^H & = \zeta  x_*y_*^H + \alpha \zeta (B_*x_*y_*^H +x_*y_*^HB_*) + \beta (B_*Nx_*y_*^H +x_*y_*^HNB_*) \\ & \quad - 8|C_*| C_* \alpha (B_*x_*y_*^H +x_*y_*^HB_*)  \;,
\end{split}
\ee
where the parameter $\zeta = (Nx_*)^Hy_*$ has to be determined by the condition \eqref{eqsc}. To this end, we cannot simply multiply both sides of \eqref{eqsc} by $r(\eps)^4$ and take the limit $\eps\nearrow\eps^{0,*}$, as in this case we obtain the undetermined condition $\Im(\bar C_* \zeta) = 0$. Instead, we must take into account also the lower order term on the left-hand side and solve the approximate equation for $r(\eps)$ small but still positive. In this way we obtain that $\zeta$ satisfies (up to errors vanishing as $\eps\nearrow\eps^{0,*}$),
\[
\zeta = \frac{\beta}{C_*\alpha} \Im(\bar C_* \zeta) \frac 1{r(\eps)} + 8|C_*| C_* \;,
\] 
whose unique solution is readily seen to be $\zeta = 8|C_*|C_*$. Plugging this value in \eqref{NK} we finally obtain that $N$ must solve the equation,
\[
N^H = 8|C_*|C_* x_*y_*^H + \beta (B_*Nx_*y_*^H +x_*y_*^HNB_*) \;, 
\]
which admits a unique solution at least for $ \|\beta B_*\|_F = \eps^{0,*}\|B_*\|_F$ sufficiently small. Moreover, by \eqref{se} and the definition of $\zeta$,
\[
\lim_{\eps\nearrow\eps^{0,*}}r(\eps)^3r''(\eps)  = -2|C_*| - \Re \left(\frac{\bar C_*}{|C_*|} \zeta\right) = -2|C_*|(1+4|C_*|)\;,
\]
that is \eqref{r3r''}. 
\endproof

\subsection*{Acknowdlegments}

The authors wish to thank the Italian  Ministry of Education, Universities and Research (MIUR) and the INdAM-\-GNCS for support.


\end{document}